\documentclass[12pt]{amsart}
\usepackage{ifthen}
\usepackage{graphicx}
\usepackage[margin=1in]{geometry}

\newtheorem{theorem}{Theorem}
\newtheorem{lemma}{Lemma}
\newtheorem{corollary}{Corollary}
\newtheorem{proposition}{Proposition}

\theoremstyle{definition}

\newtheorem{remark}{Remark}
\newtheorem*{acknowledgments}{Acknowledgments}

\theoremstyle{remark}
\newtheorem*{notation}{Notation}

\numberwithin{equation}{section}

\newcommand{\abs}[1]{\left| #1 \right|}
\newcommand{\expr}[1]{\left( #1 \right)}
\newcommand{\floor}[1]{\left\lfloor #1 \right\rfloor}
\newcommand{\norm}[1]{\left\| #1 \right\|}
\newcommand{\set}[1]{\left\{ #1 \right\}}
\newcommand{\scalar}[1]{\left\langle #1 \right\rangle}
\newcommand{\A}{\mathcal{A}}
\newcommand{\C}{\mathbf{C}}
\newcommand{\R}{\mathbf{R}}

\newcommand{\pr}{\mathbf{P}}
\newcommand{\ex}{\mathbf{E}}
\newcommand{\ind}{\mathbf{1}}
\newcommand{\eps}{\varepsilon}
\newcommand{\ph}{\varphi}
\newcommand{\ro}{\varrho}
\newcommand{\thet}{\vartheta}
\newcommand{\lap}{\mathcal{L}}
\newcommand{\sub}{\subseteq}
\newcommand{\cl}[1]{\overline{#1}}

\newcommand{\pv}{\mathop{\rm pv}}
\newcommand{\sfrac}[2]{\mbox{$\frac{#1}{#2}$}}

\providecommand{\id}{\mathop{\rm Id}\nolimits}
\providecommand{\sign}{\mathop{\rm sign}\nolimits}
\providecommand{\supp}{\mathop{\rm supp}\nolimits}
\providecommand{\res}{\mathop{\rm Res}\nolimits}

\providecommand{\real}{\mathop{\rm Re}\nolimits}
\providecommand{\imag}{\mathop{\rm Im}\nolimits}

\newcommand{\formula}[2][nolabel]
{\ifthenelse{\equal{#1}{nolabel}}
 {\begin{align*} #2 \end{align*}}
 {\ifthenelse{\equal{#1}{}}
  {\begin{align} #2 \end{align}}
  {\begin{align} \label{#1} #2 \end{align}}
 }
}

\title[Cauchy process on half-line and interval]{Spectral properties of the Cauchy process on half-line and interval}

\author{Tadeusz Kulczycki}
\address{Tadeusz Kulczycki \\ Institute of Mathematics \\ Polish Academy of
Sciences \\ ul. Kopernika 18 \\51-617 Wroclaw \\ Poland  and  Institute of Mathematics and Computer Science \\ Wroc{\l}aw University of Technology \\ ul. Wybrze{\.z}e Wyspia{\'n}\-skiego 27 \\ 50-370 Wroc{\l}aw, Poland}
\email{tkulczycki@impan.pan.wroc.pl}

\author{Mateusz Kwa{\'s}nicki}
\address{Mateusz Kwa{\'s}nicki \\ Institute of Mathematics and Computer Science \\ Wroc{\l}aw University of Technology \\ ul. Wybrze{\.z}e Wyspia{\'n}\-skiego 27 \\ 50-370 Wroc{\l}aw, Poland}
\email{mateusz.kwasnicki@pwr.wroc.pl}

\author{Jacek Ma{\l}ecki}
\address{Jacek Ma{\l}ecki \\ Institute of Mathematics and Computer Science \\ Wroc{\l}aw University of Technology \\ ul. Wybrze{\.z}e Wyspia{\'n}\-skiego 27 \\ 50-370 Wroc{\l}aw, Poland}
\email{jacek.malecki@pwr.wroc.pl}

\author{Andrzej Stos}
\address{Andrzej Stos \\ Laboratoire de Math\'ematiques\\ Universit\'e Clermont-Ferrand II \\ Campus des C\'ezeaux\\ 24 av. des Landais\\ 63177 Aubi\`ere Cedex, France }
\email{stos@math.univ-bpclermont.fr}

\thanks{The work was supported by the Polish Ministry of Science and Higher Education grant no. N~N201~373136}

\subjclass[2000]{60G52, 35J25, 35P05}

\begin{document}

\begin{abstract}
We study the spectral properties of the transition semigroup of the killed one-dimensional Cauchy process on the half-line $(0, \infty)$ and the interval $(-1, 1)$. This process is related to the square root of one-dimensional Laplacian $\A = -\sqrt{-\frac{d^2}{d x^2}}$ with a Dirichlet exterior condition (on a complement of a domain), and to a mixed Steklov problem in the half-plane. For the half-line, an explicit formula for generalized eigenfunctions $\psi_\lambda$ of $\A$ is derived, and then used to construct spectral representation of $\A$. Explicit formulas for the transition density of the killed Cauchy process in the half-line (or the heat kernel of $\A$ in $(0, \infty)$), and for the distribution of the first exit time from the half-line follow. The formula for $\psi_\lambda$ is also used to construct approximations to eigenfunctions of $\A$ in the interval. For the eigenvalues $\lambda_n$ of $\A$ in the interval the asymptotic formula $\lambda_n = \frac{n \pi}{2} - \frac{\pi}{8} + O(\frac{1}{n})$ is derived, and all eigenvalues $\lambda_n$ are proved to be simple. Finally, efficient numerical methods of estimation of eigenvalues $\lambda_n$ are applied to obtain lower and upper numerical bounds for the first few eigenvalues up to 9th decimal point.
\end{abstract}

\maketitle

%
%

\section{Introduction}

\noindent
Let $(X_t)$, $t \ge 0$, be the one-dimensional Cauchy process, that is a one-dimensional symmetric $\alpha$-stable process for $\alpha = 1$. Let us consider the Cauchy process killed upon first exit time from $D$ for $D = (0, \infty)$ and $D = (-1, 1)$. The purpose of this article is to study the spectral properties of the transition semigroup of this killed process, defined by
\formula{
  P^D_t f(x) & = \ex_x \expr{f(X_t) \; ; \; X_s \in D \; \text{for all} \; s \in [0,t]} , && f \in L^p(D) ,
}
and its infinitesimal generator $\A_D$, which is the operator $-\sqrt{-\frac{d^2}{d x^2}}$ with a Dirichlet exterior condition (on $D^c$); see the Preliminaries section for a formal introduction. The key problem in our paper is the description of eigenfunctions and eigenvalues of $\A_D$ and $P^D_t$. The study of the spectral theoretic properties of the semigroups of killed symmetric $\alpha$-stable processes has been the subject of many papers in recent years, see for example~\cite{bib:bk04,bib:bk06:spectral,bib:bk09,bib:cs97,bib:cs05,bib:cs06,bib:d04,bib:dm07}. Our paper is a continuation of the work of Ba{\~n}uelos and Kulczycki~\cite{bib:bk04}.

In the first part of the paper (Sections~\ref{sec:hl}--\ref{sec:pdt}), the identification of the spectral problem for $P^D_t$ and the so-called mixed Steklov problem in two dimensions, a method developed in~\cite{bib:bk04}, is applied for the case of the half-line $D = (0,\infty)$. Instead of searching for a function $f$ satisfying $P^D_t f(x) = e^{-\lambda t} f(x)$ for $x \in D$, $f(x) = 0$ for $x \in D^c$, we solve the equivalent mixed Steklov problem
\formula[]{
  \label{eq:intro:1} \Delta u(x, y) & = 0, && x \in \R , \, y > 0 , \\
  \label{eq:intro:2} \sfrac{\partial}{\partial y} u(x, 0) & = -\lambda u(x, 0), && x \in D , \\
  \label{eq:intro:3} u(x, 0) & = 0, && x \notin D ,
}
where $\Delta = \frac{\partial^2}{\partial x^2} + \frac{\partial^2}{\partial y^2}$ is the Laplace operator in $\R^2$.
The relation between $f$ and $u$ is here given by $u(x, y) = \ex_x f(X_y)$. In this way a nonlocal spectral problem for the pseudo-differential operator on $\R$ (or its semigroup $(P^D_t)$ on a domain~$D$) is transformed into a local one for a harmonic function of two variables, with spectral parameter in the boundary conditions. From the point of view of stochastic processes, this corresponds to the identification of the jump-type process $(X_t)$ with the trace left on the horizontal axis by the two-dimensional Brownian motion. Similar or related methods were also applied e.g. by DeBlassie and Mendez-H{\'e}rnandez~\cite{bib:d90,bib:d04,bib:dm07,bib:m02}, and the idea can be traced back to the work of Spitzer~\cite{bib:s58}, see also~\cite{bib:mo69}.

When $D = (0, \infty)$, the spectrum of $\A_D$ is equal to $(-\infty, 0]$ and is of continuous type, so there are no eigenfunctions of $\A_D$ in $L^2(D)$ (this follows easily from scaling properties of $\A_D$; see also Theorem~\ref{th:spec} below). It turns out, however, that for all $\lambda > 0$ there exist continuous generalized eigenfunctions $\psi_\lambda \in L^\infty(D)$. More precisely, we have $P^D_t \psi_\lambda = e^{-\lambda t} \psi_\lambda$. Using the identification described in the previous paragraph, an explicit formula for $\psi_\lambda$ is derived in Section~\ref{sec:hl}, see~\eqref{eq:hl:psi} and~\eqref{eq:hl:r}.

Surprisingly, to our knowledge, there are no earlier works concerning the spectral problem $P^D_t f(x) = e^{-\lambda t} f(x)$ for $x \in D$, $f(x) = 0$ for $x \in D^c$ for the half-line $D = (0,\infty)$, or the equivalent problem~\eqref{eq:intro:1}--\eqref{eq:intro:3}. However, there is an extensive literature concerning the related sloshing problem in the half-plane, i.e. the problem given by~\eqref{eq:intro:1}, \eqref{eq:intro:2} and the Neumann condition
\formula{
  \sfrac{\partial}{\partial y} u(x, 0) & = 0, && x \notin D,
}
in place of the Dirichlet one~\eqref{eq:intro:3}. The sloshing problem is one of the fundamental problems in the theory of linear water waves, see e.g.~\cite{bib:fk83} for a historical survey. The explicit solution of the sloshing problem in the half-plane for $D = (0, \infty)$ was first obtained by Friedrichs and Lewy in 1947~\cite{bib:fl47}, see also~\cite{bib:cmr05,bib:h64,bib:kk04}. Both methods and results of the Section~\ref{sec:hl} are closely related to their counterparts for the sloshing problem in the half-plane.

Sections~\ref{sec:B} and~\ref{sec:err} are rather technical and the remainder of the article relies on their results. Certain holomorphic functions play an important role in the derivation of $\psi_\lambda$, and one of these functions is studied in Section~\ref{sec:B}. In particular, the Fourier-Laplace transform of $\psi_\lambda$ is derived, see~\eqref{eq:hl:L}. The formula for $\psi_\lambda$ is of the form $\psi_\lambda(x) = \sin(\lambda x + \frac{\pi}{8}) - r(\lambda x)$, where $r$ is the Laplace transform of a positive integrable function. In Section~\ref{sec:err} we obtain estimates of the function $r$.

In Section~\ref{sec:spec} it is proved that that $\psi_\lambda$ yield a generalized eigenfunction expansion of $\A_D$ for $D = (0, \infty)$ in the sense of~\cite{bib:g59}, see e.g.~\cite{bib:psw89,bib:s82} and the references therein for similar results for differential operators. In other words, the transformation $\Pi f = \scalar{f, \psi_\lambda}$ is an isometric (up to a constant) mapping of $L^2(D)$ onto $L^2(0, \infty)$ which diagonalizes $\A_D$, $\Pi \A_D f = -\lambda \A_D f$, see Theorem~\ref{th:spec}.

The spectral decomposition and results of Section~\ref{sec:B} enable us to derive an explicit formula for the kernel function $p^D_t(x, y)$ of $P^D_t$, i.e. the transition density of the Cauchy process killed on exiting $D = (0, \infty)$ (or the heat kernel for $-\sqrt{-\frac{d^2}{d x^2}}$ with Dirichlet exterior condition on $D^c$), see Theorem~\ref{th:heat} in Section~\ref{sec:pdt}. This extends the results of~\cite{bib:bg:pre,bib:bgr:pre,bib:cks:pre}, where two-sided estimates for $p^D_t(x, y)$ were obtained (in a more general setting). As a corollary, we obtain an explicit formula for the density of the distribution of the first exit time from $(0, \infty)$, see Theorem~\ref{th:time}. This gives even a new result for 2-dimensional Brownian motion, see Corollary~\ref{cor:Brownian}. Namely, we obtain the distribution of some local time of 2-dimensional Brownian motion killed at some entrance time.

The spectral problem for the interval $D = (-1, 1)$ is studied in the second part of the article (Sections~\ref{sec:int}--\ref{sec:num}). We remark that due to translation invariance and scaling property of $(X_t)$, the results for $(-1, 1)$ extend easily to any open interval. It is well-known that there is an infinite sequence of continuous eigenfunction $\ph_n \in D$ such that $\A_D \ph_n = -\lambda_n \ph_n$ on $D$, $\ph_n \equiv 0$ on $D^c$, where $0 < \lambda_1 < \lambda_2 \le \lambda_3 \le ... \rightarrow \infty$. Each $\ph_n$ is either symmetric or antisymmetric. The study of the properties of $\ph_n$ and $\lambda_n$, dates back to the paper of Blumenthal and Getoor~\cite{bib:bg59}, where the Weyl-type asymptotic law was proved for a class of Markov processes in domains. In~\cite{bib:bg59} (formula (3.6)) it was proved that $\lambda_n/n \to \pi/2$ as $n \to \infty$. Over the last few years, there have been an increasing amount of research related to this topic see e.g.~\cite{bib:bk04,bib:bk06:spectral,bib:cs05,bib:cs06,bib:d04,bib:dm07,bib:k08,bib:k09:pre} and the references therein. In \cite{bib:bk04} it was shown that $\lambda_n \le \frac{n\pi}{2}$. The best known estimates for general $\lambda_n$, namely $\frac{n \pi}{4} \le \lambda_n \le \frac{n \pi}{2}$, were proved in~\cite{bib:cs05}, Example 5.1, where subordinate Brownian motions in bounded domains are studied. The simplicity of eigenvalues was studied in Section~5 of~\cite{bib:bk04}, where $\lambda_2$ and $\lambda_3$ are proved to be simple (simplicity of $\lambda_1$ is standard), and in~\cite{bib:k09:pre}, where all eigenvalues are proved to have at most double multiplicity. All these results are improved below.

In Section~\ref{sec:int} approximations $\tilde{\ph}_n$ to eigenfunctions $\ph_n$ are constructed by interpolating the translated eigenfunction for the half-line $\psi_\lambda(1 + x)$ and $\psi_\lambda(1 - x)$ with $\lambda = \frac{n \pi}{2} - \frac{\pi}{8}$. It is then shown that $\A_D \tilde{\ph}_n$ is nearly equal to $-\lambda \tilde{\ph}_n$. This is used in Section~\ref{sec:eigv} to prove that 
\formula{
  \left|\lambda_n -\left( \frac{n \pi}{2} - \frac{\pi}{8}\right)\right| & \le \frac{1}{n}, && n \ge 1,
}
and that the eigenvalues $\lambda_n$ are simple, see Theorem~\ref{th:int:lambda}. Finally, various properties of $\ph_n$ are shown in Section~\ref{sec:eigf}, see Corollaries~\ref{cor:sign}--\ref{cor:bounded}.

An application of numerical methods for estimation of eigenvalues to our problem is described in the last section. To get the upper bounds we use the Rayleigh-Ritz method for the Green operator, and for the lower bounds the Weinstein-Aronszajn method of intermediate problems is applied for \eqref{eq:intro:1}-\eqref{eq:intro:3}. The numerical bounds of ca. 10-digit accuracy are given by formula~\eqref{eq:num}.

Although probabilistic interpretation is the primary source of motivation, we use purely analytic arguments. In fact, the Cauchy process and related probabilistic notions are only used in Section~\ref{sec:pre} to give a concise definition of the killed semigroup $(P^D_t)$, and in Appendix~A.

%
%

\section{Notation and preliminaries}
\label{sec:pre}

\noindent
We begin with a brief introduction to the Cauchy process $(X_t)$ and its relation to the Steklov problem. We only collect the properties used in the sequel; for a more detailed exposition the reader is referred to~\cite{bib:bk04} or~\cite{bib:cs97,bib:k98}. For an introduction to more general Markov processes, see e.g.~\cite{bib:bg68,bib:d65,bib:s99}. In the final part of this section, basic facts concerning the Fourier transform, the Hilbert transform and Paley-Wiener theorems are recalled.

The one-dimensional Cauchy process $(X_t)$ is the symmetric $1$-stable process, that is, the L{\'e}vy process with one-dimensional distributions
\formula{
  \pr_x(X_t \in dy) & = p_t(y - x) dy = \frac{1}{\pi} \frac{t}{t^2 + (y - x)^2} \, dy .
}
Here $\pr_x$ corresponds to the process starting at $x \in \R$; $\ex_x$ is the expectation with respect to $\pr_x$. Clearly, the $\pr_x$-distributions of $(X_t + a)$ and $(b X_t)$ are equal to $\pr_{x+a}$-distribution of $(X_t)$ and $\pr_{b x}$-distribution of $(X_{b t})$ respectively; these are the translation invariance and scaling property mentioned in the Introduction. The transition semigroup of $(X_t)$ is defined by
\formula{
  P_t f(x) & = \ex_x f(X_t) = f * p_t(x) , && f \in L^p(\R) , \, p \in [1, \infty] , \, t > 0 ,
}
and $P_0 f(x) = f(x)$. This is a contraction semigroup on each $L^p(\R)$, $p \in [1, \infty]$, strongly continuous if $p \in [1, \infty)$, and when $f$ is continuous and bounded, then $P_t f$ converges to $f$ locally uniformly as $t \searrow 0$. The infinitesimal generator $\A$ of $(P_t)$ acting on $L^2(\R)$ is the square root of the second derivative operator. More precisely, for a smooth function $f$ with compact support we have
\formula{
  \A f(x) & = -\sqrt{-\sfrac{d^2}{d x^2}} \, f(x) = \frac{1}{\pi} \pv\int_{-\infty}^\infty \frac{f(y) - f(x)}{(y - x)^2} \, dy ,
}
where the integral is the Cauchy principal value.

Throughout this article, $D$ always denotes the interval $(-1, 1)$ or the half-line $(0, \infty)$. The time of the first exit from $D$ is defined by
\formula{
  \tau_D & = \inf \set{t \ge 0 \; : \; X_t \notin D} ,
}
and the semigroup of the process $(X_t)$ killed at time $\tau_D$ is given by
\formula{
  P^D_t f(x) & = \ex_x \expr{f(X_t) \; ; \; X_s \in D \; \text{for all} \; s \in [0,t]} = \ex_x \expr{f(X_t) \; ; \; t < \tau_D} ,
}
where $t \ge 0$. This is again a well-defined contraction semigroup on every $L^p(D)$ space, $p \in [1, \infty]$, strongly continuous if $p \in [1, \infty)$. If $f$ continuous and bounded in $\R$ and vanishes in $(-\infty, 0]$, then $P^D_t f$ converges to $f$ locally uniformly as $t \searrow 0$. The semigroup $(P^D_t)$ admits a jointly continuous kernel function $p^D_t(x, y)$ ($t > 0$, $x, y \in D$); clearly, $p^D_t(x, y) \le p_t(y - x) \le \frac{1}{\pi t}$. By $\A_D$ we denote the infinitesimal generator of $(P^D_t)$ acting on $L^2(D)$. Since this is a Friedrichs extention on $L^2(D)$ of $\A$ restricted to the class of smooth functions supported in a compact subset of $D$, we sometimes say that $\A_D$ is the square root of Laplacian with Dirichlet exterior conditions (on $D^c$).

Let us describe in more details the connection between the spectral problem for the semigroup $(P_t^D)$ and the mixed Steklov problem~\eqref{eq:intro:1}--\eqref{eq:intro:3}, established in \cite{bib:bk04}. The main idea is to consider the harmonic extension $u(x, y)$ of a function $f$ to the upper half-plane $x \in \R$, $y > 0$. Let $f \in L^p(\R)$ for some $p \in [1, \infty]$, and define 
\formula{
  u(x, y) & = P_y f(x) = \frac{1}{\pi} \int_{-\infty}^{\infty} \frac{y}{y^2 + (z - x)^2} \, f(z) dz .
}
Then $u$ is harmonic in the upper half-plane $\R \times (0,\infty)$, and if $p \in [1, \infty)$, then $u(\cdot, y)$ converges to $f$ in $L^p(\R)$ as $y \searrow 0$. Conversely, for $p \in (1, \infty)$, if $u(x, y)$ is harmonic in the upper half-plane and the $L^p(\R)$ norms of $u(\cdot, y)$ are bounded for $y > 0$, then $u(\cdot, y)$ converges in $L^p(\R)$ to some $f$ when $y \searrow 0$, and $u(x, y) = P_y f(x)$. By the definition,
\formula{
  \frac{\partial}{\partial y} \, u(x, 0) = \lim_{y \searrow 0} \frac{P_y f(x) - f(x)}{y}
}
pointwise for all $x \in \R$. When $f$ is in the domain of $\A$, then the above limit exists in $L^2(\R)$ and it is equal to $\A f$.

Our motivation to study the mixed Steklov problem~\eqref{eq:intro:1}--\eqref{eq:intro:3} comes from the following simple extension of Theorem~1.1 in~\cite{bib:bk04} to the case of unbounded domains. A partial converse is given in the proof of Theorem~\ref{th:halfline} in Section~\ref{sec:hl}.

\begin{proposition}
\label{prop:equivalence}
Let $D = (0, \infty)$ and $\lambda > 0$. Suppose that $f : \R \to \R$ is continuous and bounded, $f(x) = 0$ for $x \le 0$, and $u(x, y) = P_y f(x)$. If $P^D_t f(x) = e^{-\lambda t} f(x)$ for all $x \in D$, $t > 0$, then $u$ satisfies~\eqref{eq:intro:1}--\eqref{eq:intro:3}.
\end{proposition}

\begin{proof}
Formulas~\eqref{eq:intro:1} and~\eqref{eq:intro:3} hold true by the definition of $u$. Since $P^D_y f(x) = e^{-\lambda y} f(x)$, we have
\formula{
  \frac{u(x, y) - u(x, 0)}{y} & = \frac{P_y f(x) - f(x)}{y} = \frac{e^{-\lambda y} - 1}{y} \, f(x) - \frac{P_y f(x) - P^D_y f(x)}{y} \, .
}
As $y \searrow 0$, the first summand converges to $-\lambda f(x)$. The second one is estimated using formula~\eqref{eq:C:pDt} from Appendix~A (see also formula~(2.9) in~\cite{bib:bk04}). If $0 < y < x$, we have
\formula{
  \abs{\frac{P_y f(x) - P^D_y f(x)}{y}} & \le \int_0^\infty \frac{p_y(z - x) - p^D_y(x, z)}{y} \, |f(z)| dz \\
  & \le \frac{\norm{f}_\infty}{\pi} \int_0^\infty \min \expr{\frac{1}{x^2}, \frac{y}{x^2 z}, \frac{y}{x z^2}} dz = \frac{y (2 + \log \frac{x}{y}) \norm{f}_\infty}{\pi x^2} ,
}
and this tends to $0$ as $y \searrow 0$. Therefore,~\eqref{eq:intro:2} is also satisfied.
\end{proof}

Finally, we briefly recall some standard facts and definitions. The Fourier transform of a (complex-valued) function $f \in L^1(\R)$ is given by $\hat{f}(x) = \int f(t) e^{-i t x} dt$; this can be continuously extended to $L^p(\R)$ whenever $1 \le p < \infty$. For $p \in (1, \infty)$, the Hilbert transform of $f \in L^p(\R)$, denoted $H f$, satisfies $(H f)\hat{\;}(t) = -\hat{f}(t) (i \sign t)$. This is a bounded linear operator on $L^p(\R)$, and for almost all $t$,
\formula[eq:pre:H]{
  H f(t) & = \frac{1}{\pi} \pv\int_{-\infty}^\infty \frac{f(s)}{t - s} ds .
}
If $f$ is H{\"o}lder continuous, then the above formula holds for all $t \in \R$ and $H f$ is continuous, see e.g.~\cite{bib:s93}.

Let $\C_+ = \set{z \in \C \; : \; \imag z > 0}$ and $\cl{\C}_+ = \set{z \in \C \; : \; \imag z \ge 0}$; $\C_-$ and $\cl{\C}_-$ are defined in a similar manner. Let $1 < p < \infty$. If $F$ is in the (complex) Hardy space $H^p(\C_+)$, i.e. $F$ is holomorphic in $\C_+$ and the $L^p(\R)$ norms of $F(\cdot + i \eps)$ are bounded in $\eps > 0$, then, as $\eps \searrow 0$, $F(\cdot + i \eps)$ converges in $L^p(\R)$ to some $f \in L^p(\R)$, which is said to be the boundary limit of $F$. In this case
\formula[eq:pre:hilbert]{
  \imag f & = H(\real f) && \text{and} & \real f & = -H(\imag f) .
}
We also have
\formula[eq:pre:hilbert:sym]{
  H \tilde{f}(t) & = -H f(-t) , & \text{where } \tilde{f}(t) & = f(-t) .
}
The following version of Paley-Wiener theorem is important in the sequel, see e.g.~\cite{bib:d70}. For $p \in (1, \infty)$, a function $f \in L^p(\R)$ is a boundary limit of some function $F \in H^p(\C_+)$ if and only if $\hat{f}$ vanishes in $(-\infty, 0)$. In this case
\formula[eq:pre:paley]{
  F(z) & = \frac{1}{2 \pi} \int_0^\infty \hat{f}(x) e^{i z x} dx , && z \in \C_+ .
}

%
%

\section{Spectral problem in the half-line}
\label{sec:hl}

\begin{notation}
To facilitate reading, in this section we strive to use the following convention. We use small letters to denote functions of the real variable and capital letters for functions on the upper half-plane $\C_+$. Real-valued functions are denoted by Greek letters, whereas Latin letters are used for complex-valued functions.
\end{notation}

We study the eigenproblem~\eqref{eq:intro:1}--\eqref{eq:intro:3} for the half-line $D = (0, \infty)$ using methods which were earlier applied to the sloshing problem with semi-infinite dock, see~\cite{bib:cmr05,bib:fl47,bib:h64}. The solution $u$ is given as the imaginary part of a holomorphic function $F$ of a complex variable $z = x + i y$, $x \in \R$, $y \ge 0$. Such a function is automatically harmonic, hence~\eqref{eq:intro:1} is satisfied. Using the Cauchy-Riemann equations, we may restate~\eqref{eq:intro:2} and~\eqref{eq:intro:3} in the following equivalent form:
\formula[]{
  \label{eq:hl:1} & \imag (i F'(x) + \lambda F(x)) = 0 & x > 0 , \\
  \label{eq:hl:2} & \imag F(x) = 0 & x \le 0 .
}
Observe that for all $\thet \in \R$ and $t < 0$, the bounded holomorphic functions
\formula{
  F(z) & = e^{i \lambda z + i \thet}, \\
  F(z) & = e^{t \lambda z - i \arctan t}
}
satisfy~\eqref{eq:hl:1}, and for for all $t > 0$ the bounded holomorphic function
\formula{
  F(z) & = e^{t \lambda z}
}
satisfies~\eqref{eq:hl:2}. This suggests searching a solution of the form:
\formula[]{
  \label{eq:hl:sol:1} F(z) & = e^{i \lambda z + i \thet} - \int_{-\infty}^0 \ro(t) e^{t \lambda z - i \arctan t} dt , & & \real z \ge 0 , \, \imag z \ge 0 , \\
  \label{eq:hl:sol:2} F(z) & = \int_0^\infty \ro(t) e^{t \lambda z} dt , & & \real z \le 0 , \, \imag z \ge 0 ,
}
where $\ro$ is an unknown real function, say in some $L^p(\R)$, $p \in (1, \infty)$, and $\thet \in \R$. The values of $F$ given by~\eqref{eq:hl:sol:1} and~\eqref{eq:hl:sol:2} must agree when $\real z = 0$, $\imag z \ge 0$, that is,
\formula{
  \int_{-\infty}^\infty e^{i \chi(t)} \ro(t) e^{i t \lambda y} dt & = e^{-\lambda y + i \thet} , && y > 0 ,
}
where $\chi(t) = \arctan t_- = \arctan(\max(-t, 0))$. Replacing $\lambda y$ by $-s$ yields that
\formula[eq:hl:glue]{
   \int_{-\infty}^\infty e^{i \chi(t)} \ro(t) e^{-i t s} dt & = e^{s + i \thet} , && s < 0 .
}
The right-hand side is the Fourier transform of $g(t) = \frac{e^{i \thet}}{2 \pi} \frac{1}{1 + i t}$. Therefore, formula~\eqref{eq:hl:glue} is equivalent to the condition:
\formula[eq:hl:pw]{
  \text{the function $a(t) = e^{i \chi(t)} \ro(t) - g(t)$ satisfies $\hat{a}(s) = 0$ for $s < 0$.}
}
Note that both $\ro$ and $g$ are in $L^p(\R)$, so that $\hat{a}$ is well-defined and $\hat{a} \in L^p(\R)$. The foregoing remarks can be summarized as follows: any real function $\ro \in L^p(\R)$ satisfying~\eqref{eq:hl:pw} yields a solution to the problem~\eqref{eq:hl:1}--\eqref{eq:hl:2}.

By Paley-Wiener theorem,~\eqref{eq:hl:pw} is satisfied if and only if $a$ is the boundary limit of a unique function $A$ in the Hardy space $H^p(\C_+)$ in the upper half-plane $\C_+ = \set{z \in \C \; : \; \imag z > 0}$. Such a function $A$ can be derived as follows. Later in this section (formula~\eqref{eq:hl:h2}; see also Appendix~B), a function $B$ holomorphic in $\C_+$ and continuous on $\cl{\C}_+$ is defined, such that $i \chi(t) - B(t) \in \R$ for all $t \in \R$. The function
\formula{
  e^{-B(t)} a(t) & = e^{i \chi(t) - B(t)} \ro(t) - e^{-B(t)} g(t)
}
is therefore the boundary limit of $e^{-B(z)} A(z)$. Note that $e^{i \chi(t) - B(t)}$ is real. The function $g(t) = \frac{e^{i \thet}}{2 \pi} \frac{1}{1 + i t}$ is the boundary limit of a meromorphic function $G(z) = \frac{e^{i \thet}}{2 \pi} \frac{1}{1 + i z}$. The function $G$ has a simple pole at $i$, so that $G(z) (e^{-B(z)} - e^{-B(i)})$ is holomorphic in $\C_+$. It follows that
\formula[eq:hl:ro01]{
  e^{-B(t)} a(t) + g(t) (e^{-B(t)} - e^{-B(i)}) & = e^{i \chi(t) - B(t)} \ro(t) - e^{-B(i)} g(t)
}
is a boundary limit of
\formula{
  \tilde{A}(z) & = e^{-B(z)} A(z) + G(z) (e^{-B(z)} - e^{-B(i)}) , && z \in \C_+ .
}
Since $G$ and $A$ are in $H^p(\C_+)$, and $|e^{-B(z)}|$ is bounded (see~\eqref{eq:aux:Best}), we must have $\tilde{A} \in H^p(\C_+)$. Let $\tilde{G}(z) = \frac{e^{-i \thet}}{2 \pi} \frac{1}{1 - i z}$. Note that by~\eqref{eq:hl:ro01}, the boundary limit of the function $\tilde{A}(z) - e^{-\overline{B(i)}} \tilde{G}(z)$ (belonging to $H^p(\C_+)$) is equal to
\formula[eq:hl:ro02]{
  e^{i \chi(t) - B(t)} \ro(t) - e^{-B(i)} g(t) - \overline{e^{-B(i)} g(t)} ,
}
which is real for all $t \in \R$. The real part of the boundary limit of an $H^p(\C_+)$ function is the negative of the Hilbert transform of its imaginary part. Therefore, the function defined by~\eqref{eq:hl:ro02} is the Hilbert transform of the constant $0$, and so it is identically $0$. It follows that
\formula[eq:hl:ro]{
  e^{i \chi(t) - B(t)} \ro(t) & = e^{-B(i)} g(t) + \overline{e^{-B(i)} g(t)} = 2 \real \expr{e^{-B(i)} g(t)} , && t \in \R .
}
Also, $\tilde{A}(z) - e^{-\overline{B(i)}} \tilde{G}(z)$ has a boundary limit $0$, so it is identically zero in $\C_+$. Hence, for $z \in \C_+$,
\formula{
  A(z) & = e^{B(z)} \expr{\tilde{A}(z) - G(z) (e^{-B(z)} - e^{-B(i)})} \\
  & = \frac{e^{-i \thet}}{2 \pi} \frac{e^{B(z) - \overline{B(i)}}}{1 - i z} - \frac{e^{i \thet}}{2 \pi} \frac{1 - e^{B(z) - B(i)}}{1 + i z} \, .
}
Since $|e^{B(z)}|$ is bounded by a constant multiple of $1 + \sqrt{|z|}$ (see~\eqref{eq:aux:Best}), $A$ defined by the above formula is in $H^p(\C_+)$ for any $p \in (2, \infty)$, and $\ro$ given by~\eqref{eq:hl:ro} is in $L^p(\R)$. Also, the boundary limit of $A$ is the function $a$ defined in~\eqref{eq:hl:pw} (this can be verified e.g. by a direct calculation), so that $\ro$ indeed is a solution to~\eqref{eq:hl:pw}.

We now come to the construction of the function $B$. We want it to be holomorphic in $\C_+$ and continuous in $\cl{\C}_+$, and $i \chi(t) - B(t)$ is to be real for all $t \in \R$. Therefore,
\formula[eq:hl:B]{
  \imag B(t) & = \chi(t) = \arctan(t_-) , && t \in \R .
}
Clearly $B$ is not in $H^p(\C_+)$, so that $\real B(t)$ cannot be expressed directly as the Hilbert transform of $\imag B(t) = \chi(t)$. We can, however, apply the Hilbert transform to $\imag B'(t) = \chi'(t)$, which is an $L^2(\R)$ function. It follows that
\formula{
  \real B'(t) & = -H(\imag B')(t) = \frac{1}{\pi} \pv\int_{-\infty}^0 \frac{1}{(t - s)(1 + s^2)} \, ds , && t \in \R ,
}
the integral on the right-hand side being the Cauchy principal value for $t < 0$. This equation is studied in Appendix~B. It follows that up to an additive constant, which we choose to be zero, we have $\real B(t) = \eta(t)$, where $\eta$ is given by~\eqref{eq:aux:int0}. By~\eqref{eq:aux:int1} and~\eqref{eq:aux:int2}, for all $t \in \R$,
\formula[eq:hl:h1]{
\begin{split}
  B(t) = i \chi(t) + \eta(t)& = i \arctan(t_-) + \log \sqrt[4]{1 + t^2} - \frac{1}{\pi} \int_0^t \frac{\log |s|}{1 + s^2} \, ds \\
  & = i \arctan(t_-) + \frac{1}{\pi} \int_{-\infty}^0 \frac{\log |t - s|}{1 + s^2} \, ds .
\end{split}
}
This formula is easily extended to complex arguments, whenever $\imag z \ge 0$, we have
\formula[eq:hl:h2]{
  B(z) = \frac{1}{\pi} \int_{-\infty}^0 \frac{\log (z - s)}{1 + s^2} \, ds ,
}
provided that the continuous branch of $\log$ is chosen on the upper half-plane $\cl{\C}_+$ (i.e. the principal branch with $\log s = \log |s| + i \pi$ for $s < 0$). We emphasize that~\eqref{eq:hl:h1} and~\eqref{eq:hl:h2} agree for $z = t < 0$ (see also Section~\ref{sec:B} and Appendix~B).

For the explicit formula for $\ro$, $B(i)$ needs to be computed. By~\eqref{eq:aux:pi8} and~\eqref{eq:aux:sqrt2},
\formula[eq:hl:hi]{
\begin{split}
  B(i) & = \frac{1}{\pi} \int_{-\infty}^0 \frac{\log (i - s)}{1 + s^2} \, ds = \frac{1}{\pi} \int_0^\infty \frac{\log (i + s)}{1 + s^2} \, ds \\
  & = \frac{1}{2 \pi} \int_0^\infty \frac{\log (1 + s^2)}{1 + s^2} \, ds + \frac{i}{\pi} \int_0^\infty \frac{\frac{\pi}{2} - \arctan s}{1 + s^2} \, ds = \frac{\log 2}{2} + \frac{i \pi}{8} \, .
\end{split}
}
Now~\eqref{eq:hl:ro} yields that
\formula{
  \ro(t) & = 2 e^{B(t) - i \chi(t)} \real \expr{e^{-B(i)} g(t)} = 2 e^{\eta(t)} \real \expr{\frac{e^{i (\thet - \frac{\pi}{8})}}{2 \pi \sqrt{2}} \frac{1}{1 + i t}} \\
  & = \frac{\sqrt{2}}{2 \pi} e^{\eta(t)} \frac{\cos (\thet - \frac{\pi}{8}) + t \sin (\thet - \frac{\pi}{8})}{1 + t^2} \, , && t \in \R .
}
Since $\thet \in \R$ is arbitrary, we conclude that there are two linearly independent solutions for $\ro$, corresponding to $\thet = \frac{\pi}{8}$ and $\thet = \frac{5 \pi}{8}$ respectively,
\formula{
  \ro(t) & = \frac{\sqrt{2}}{2 \pi} \frac{1}{1 + t^2} \, e^{\eta(t)} , & \text{and} && \tilde{\ro}(t) & = \frac{\sqrt{2}}{2 \pi} \frac{t}{1 + t^2} \, e^{\eta(t)} .
}
The solution to~\eqref{eq:hl:1}--\eqref{eq:hl:2} corresponding to $\thet = \frac{\pi}{8}$ and $\ro$ as above is therefore given by
\formula[]{
  \label{eq:hl:F:1a}
  F(z) & = e^{i \lambda z + i \frac{\pi}{8}} - \frac{\sqrt{2}}{2 \pi} \int_{-\infty}^0 \frac{1}{1 + t^2} \, e^{\eta(t)} e^{t \lambda z - i \arctan t} dt , & & \real z \ge 0 , \, \imag z \ge 0 \\
  \label{eq:hl:F:2}
  F(z) & = \frac{\sqrt{2}}{2 \pi} \int_0^\infty \frac{1}{1 + t^2} \, e^{\eta(t)} e^{t \lambda z} dt , & & \real z \le 0 , \, \imag z \ge 0 .
}
By~\eqref{eq:aux:intest}, we have $\ro \in L^1(\R)$, and so $F$ is bounded and continuous. Furthermore, it can be easily verified that the solution corresponding to $\thet = \frac{5 \pi}{8}$ and $\tilde{\ro}$ is given by $\frac{F'(z)}{\lambda}$. Since $\tilde{\ro}$ decays at infinity as $|t|^{-\frac{1}{2}}$, $F'(z)$ has a singularity of order $|z|^{-\frac{1}{2}}$ at zero and it is not bounded near $0$. For that reason, in the sequel we only study the solution $F(z)$ given by~\eqref{eq:hl:F:1a}--\eqref{eq:hl:F:2}.

Since $e^{-i \arctan t} = (1 - i t) / \sqrt{1 + t^2}$ and $e^{\eta(t)} = e^{-\eta(-t)} \sqrt{1 + t^2}$ (see~\eqref{eq:aux:etasym}), we can rewrite~\eqref{eq:hl:F:1a} as
\formula[eq:hl:F:1]{
  F(z) & = e^{i \lambda z + i \frac{\pi}{8}} - \frac{\sqrt{2}}{2 \pi} \int_0^\infty \frac{1 + i t}{1 + t^2} \, e^{-\eta(t)} e^{-t \lambda z} dt , & & \real z \ge 0 , \, \imag z \ge 0 .
}
Therefore, we proved the following theorem.

\begin{theorem}
\label{th:halfspace}
The bounded solution of~\eqref{eq:intro:1}--\eqref{eq:intro:3} for $D = (0, \infty)$ is given by
\formula[eq:hl:u1]{
\begin{split}
  u(x, y) & = e^{-\lambda y} \sin(\lambda x + \sfrac{\pi}{8}) \\
  & \hspace*{-5mm} - \frac{\sqrt{2}}{2 \pi} \int_0^\infty \frac{t \cos(t \lambda y) - \sin(t \lambda y)}{1 + t^2} \exp\expr{-\frac{1}{\pi} \int_0^\infty \frac{\log (t + s)}{1 + s^2} ds} e^{-t \lambda x} dt
\end{split}
}
for $x \ge 0$, $y \ge 0$, and
\formula[eq:hl:u2]{
  u(x, y) & = \frac{\sqrt{2}}{2 \pi} \int_0^\infty \frac{\sin(t \lambda y)}{1 + t^2} \exp\expr{\frac{1}{\pi} \int_0^\infty \frac{\log (t + s)}{1 + s^2} ds} e^{t \lambda x} dt
}
for $x \le 0$ and $y \ge 0$.
\end{theorem}

The main result of this section, stated below, follows from Theorem~\ref{th:halfspace} and a partial converse to Proposition~\ref{prop:equivalence}.

\begin{figure}
\centering
\begin{tabular}{cc}
\includegraphics[width=5cm]{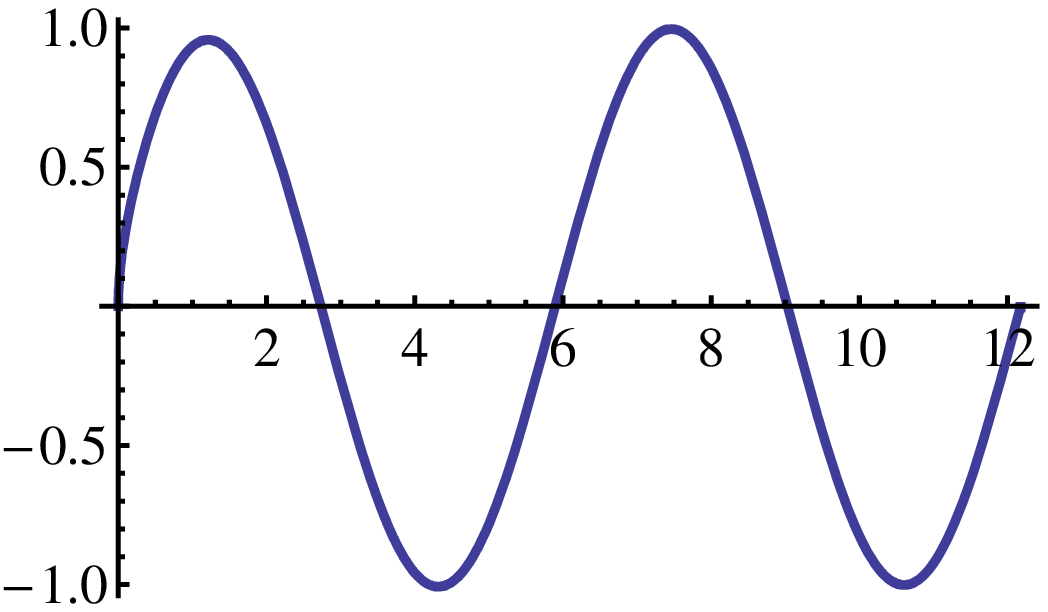} & \includegraphics[width=5cm]{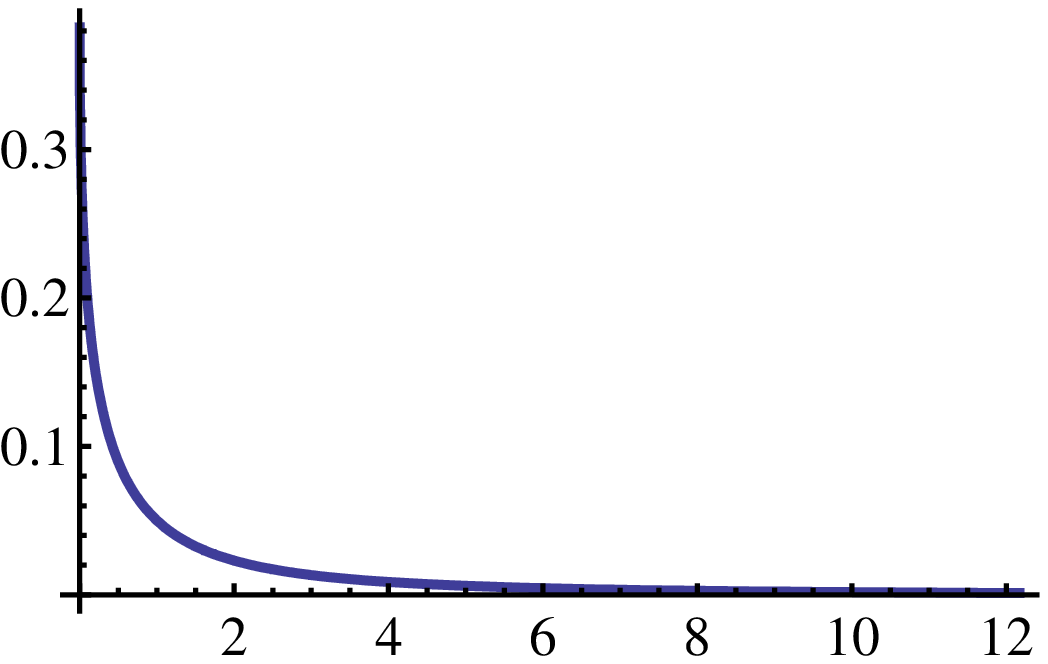} \\
(a) & (b) \\
\end{tabular}
\caption{(a) Graph of $\psi_1$; (b) Graph of the remainder term $r(x) = \sin(x + \frac{\pi}{8}) - \psi_1(x)$}
\end{figure}

\begin{theorem}
\label{th:halfline}
Let $D = (0, \infty)$. For $\lambda > 0$, the function
\formula[eq:hl:psi]{
  \psi_\lambda(x) & = \sin(\lambda x + \sfrac{\pi}{8}) - r_\lambda(x), && x > 0 ,
}
where
\formula[eq:hl:r]{
  r_\lambda(x) & = r(\lambda x) = \frac{\sqrt{2}}{2 \pi} \int_0^\infty \frac{t}{1 + t^2} \exp\expr{-\frac{1}{\pi} \int_0^\infty \frac{\log (t + s)}{1 + s^2} ds} e^{-t \lambda x} dt ,
}
is the eigenfunction of the semigroup $(P^D_t)$ acting on $C(D)$ corresponding to eigenvalue $\lambda$.
\end{theorem}

\begin{proof}
With the notation of Theorem~\ref{th:halfspace}, we have $\psi_\lambda(x) = u(x, 0)$; we extend $\psi_\lambda$ to be $0$ on $(-\infty, 0]$. Since $u$ is harmonic and bounded in the upper half-plane, we have $P_y \psi_\lambda(x) = u(x, y)$ ($y > 0$, $x \in \R$). Since $u$ satisfies~\eqref{eq:intro:2}, for all $x > 0$, $\frac{1}{y} (P_y \psi_\lambda(x) - \psi_\lambda(x))$ converges to $-\lambda \psi_\lambda(x)$ as $y \searrow 0$. We will now prove (formula~\eqref{eq:hl:genpsi:dom1}) that this convergence is dominated by an appropriate function.

Below we assume that $\lambda > 0$, $x > 0$ and $0 < y < \frac{1}{\lambda}$. By formula~\eqref{eq:hl:u1}, we have
\formula[eq:hl:genpsi]{
\begin{split}
  \frac{P_y \psi_\lambda(x) - e^{-\lambda y} \psi_\lambda(x)}{y} & = \\
  & \hspace*{-15mm} - \frac{\sqrt{2}}{2 \pi} \int_0^\infty \frac{t \cos(t \lambda y) - \sin(t \lambda y) - e^{-\lambda y} t}{(1 + t^2) y} e^{-\eta(t)} e^{-t \lambda x} dt .
\end{split}
}
Since $|1 - \cos z| \le \frac{z^2}{2}$, $|z - \sin z| \le \frac{z^3}{3}$, $|1 - z - e^{-z}| \le \frac{z^2}{2}$ and $\lambda y < 1$, we have
\formula{
  \abs{t \cos(t \lambda y) - \sin(t \lambda y) - e^{-\lambda y} t} & \le \lambda^2 t \expr{\frac{t^2}{2} + \frac{t^2 \lambda y}{3} + \frac{1}{2}} y^2 < \lambda^2 t (1 + t^2) y^2 \, .
}
Using also $e^{-\eta(t)} \le e^{\mathcal{C} / \pi} (1 + t^2)^{-\frac{1}{4}}$ (see~\eqref{eq:aux:est}), and then $(1 + t^2) \ge t^2$, we obtain
\formula[eq:hl:genpsi2]{
\begin{split}
  \abs{\int_0^{\frac{1}{\lambda y}} \frac{t \cos(t \lambda y) - \sin(t \lambda y) - e^{-\lambda y} t}{(1 + t^2) y} e^{-\eta(t)} e^{-t \lambda x} dt} \hspace*{-70mm} & \\
  & \le e^{\frac{\mathcal{C}}{\pi}} \lambda^2 y \int_0^{\frac{1}{\lambda y}} \frac{t (1 + t^2)}{(1 + t^2)^{\frac{5}{4}}} \, e^{-t \lambda x} dt \le e^{\frac{\mathcal{C}}{\pi}} \lambda^2 y \int_0^{\frac{1}{\lambda y}} \sqrt{t} \, e^{-t \lambda x} dt .
\end{split}
}
Furthermore,
\formula[eq:hl:genpsi3]{
  \int_0^{\frac{1}{\lambda y}} \sqrt{t} \, e^{-t \lambda x} dt & \le (\lambda y)^{-\frac{3}{4}} \int_0^\infty t^{-\frac{1}{4}} \, e^{-t \lambda x} dt \le \frac{\Gamma(\frac{3}{4})}{(\lambda^2 x y)^{\frac{3}{4}}} \, .
}
In a similar manner, but using $|t \cos(t \lambda y) - \sin(t \lambda y) - e^{-\lambda y} t| \le t + t \lambda y + t \le 3 t$, we obtain
\formula[eq:hl:genpsi4]{
\begin{split}
  \abs{\int_{\frac{1}{\lambda y}}^\infty \frac{t \cos(t \lambda y) - \sin(t \lambda y) - e^{-\lambda y} t}{(1 + t^2) y} e^{-\eta(t)} e^{-t \lambda x} dt} \hspace*{-70mm} & \\
  & \le 3 e^{\frac{\mathcal{C}}{\pi}} \int_{\frac{1}{\lambda y}}^\infty \frac{t}{(1 + t^2)^{\frac{5}{4}} y} e^{-t \lambda x} dt \le \frac{3 e^{\frac{\mathcal{C}}{\pi}}}{y} \int_{\frac{1}{\lambda y}}^\infty t^{-\frac{3}{2}} e^{-t \lambda x} dt ,
\end{split}
}
and
\formula[eq:hl:genpsi5]{
  \int_{\frac{1}{\lambda y}}^\infty t^{-\frac{3}{2}} e^{-t \lambda x} dt & \le (\lambda y)^{\frac{5}{4}} \int_0^\infty t^{-\frac{1}{4}} e^{-t \lambda x} dt = \frac{\Gamma(\frac{3}{4}) (\lambda y)^{\frac{5}{4}}}{(\lambda x)^{\frac{3}{4}}} \, .
}
Formulas~\eqref{eq:hl:genpsi}--\eqref{eq:hl:genpsi5} yield, after simplification, that
\formula[eq:hl:genpsi:dom1]{
\begin{split}
  \abs{\frac{P_y \psi_\lambda(x) - e^{-\lambda y} \psi_\lambda(x)}{y}} & \le \frac{2 \sqrt{2} e^{\frac{\mathcal{C}}{\pi}} \Gamma(\frac{3}{4}) \lambda^{\frac{1}{2}} y^{\frac{1}{4}}}{\pi x^{\frac{3}{4}}} = \frac{c_1(\lambda) y^{\frac{1}{4}}}{x^{\frac{3}{4}}}
\end{split}
}
with some constant $c_1(\lambda)$.

We are now going to replace $P_y$ by $P^D_y$ in~\eqref{eq:hl:genpsi:dom1}. In Section~\ref{sec:err} it is proved (using only the definition~\eqref{eq:hl:psi} and~\eqref{eq:hl:r} of $\psi_\lambda$) that $|\psi_\lambda(x)| = |\psi_1(\lambda x)| \le 2 \sqrt{\lambda x}$, see~\eqref{eq:err:psisqrt}. This and~\eqref{eq:C:pDt} yield that for $0 < y < x$ we have
\formula[eq:hl:genpsi6]{
\begin{split}
  \abs{\frac{P_y \psi_\lambda(x) - P^D_y \psi_\lambda(x)}{y}} & \le \int_0^\infty \frac{p_y(z - x) - p^D_y(z - x)}{y} \, |\psi_\lambda(z)| dz \\
  & \le \frac{2 \sqrt{\lambda}}{\pi} \int_0^\infty \min \expr{\frac{1}{x^2}, \frac{y}{x^2 z}, \frac{y}{x z^2}} \sqrt{z} \, dz \\
  & = \frac{8 \sqrt{\lambda} \, y (3 \sqrt{x} - \sqrt{y})}{3 \pi x^2} \le \frac{8 \sqrt{\lambda} y}{\pi x^{\frac{3}{2}}} \le \frac{8 \sqrt{\lambda} y^{\frac{1}{4}}}{\pi x^{\frac{3}{4}}} \, .
\end{split}
}
When $0 < x < y$, in a similar manner
\formula[eq:hl:genpsi7]{
\begin{split}
  \abs{\frac{P_y \psi_\lambda(x) - P^D_y \psi_\lambda(x)}{y}} & \le \frac{2 \sqrt{\lambda}}{\pi} \int_0^\infty \min \expr{\frac{1}{y^2}, \frac{1}{z^2}} \sqrt{z} \, dz \\
  & = \frac{16 \sqrt{\lambda}}{3 \pi \sqrt{y}} \le \frac{16 \sqrt{\lambda} y^{\frac{1}{4}}}{3 \pi x^{\frac{3}{4}}} \, .
\end{split}
}
Finally, by~\eqref{eq:hl:genpsi:dom1}, \eqref{eq:hl:genpsi6} and~\eqref{eq:hl:genpsi7}, there is a constant $c_2(\lambda)$ such that
\formula[eq:hl:genpsi:dom2]{
  \abs{\frac{P^D_y \psi_\lambda(x) - e^{-\lambda y} \psi_\lambda(x)}{y}} & \le \frac{c_2(\lambda) y^{\frac{1}{4}}}{x^{\frac{3}{4}}} \, .
}
For any fixed $x > 0$ and $t > 0$, the one-sided derivative of $e^{\lambda t} P^D_t \psi_\lambda(x)$ with respect to $t$ equals
\formula{
  \frac{\partial}{\partial t_+} \expr{e^{\lambda t} P^D_t \psi_\lambda(x)} & = \lim_{y \searrow 0} \frac{e^{\lambda (t + y)} P^D_{t + y} \psi_\lambda(x) - e^{\lambda t} P^D_t \psi_\lambda(x)}{y} \\
  & = \lim_{y \searrow 0} e^{\lambda (t + y)} \int_0^\infty p^D_t(x, z) \, \frac{P^D_y \psi_\lambda(z) - e^{-\lambda y} \psi_\lambda(z)}{y} \, dz .
}
Since $p^D_t(x, z) \le p_t(z - x) \le \frac{1}{t}$, by~\eqref{eq:hl:genpsi:dom2} we have
\formula{
  & \abs{\int_0^\infty p^D_t(x, z) \, \frac{P^D_y \psi_\lambda(z) - e^{-\lambda y} \psi_\lambda(z)}{y} \, dz} \le \int_0^\infty \frac{c_2(\lambda) y^{\frac{1}{4}}}{z^{\frac{3}{4}}} \, p^D_t(x, z) dz \\
  & \hspace{20mm} \le c_2(\lambda) y^{\frac{1}{4}} \expr{\int_1^\infty p^D_t(x, z) dz + \frac{1}{t} \int_0^1 \frac{1}{z^{\frac{3}{4}}} \, dz} \le \expr{1 + \frac{4}{t}} c_2(\lambda) y^{\frac{1}{4}} .
}
The right hand side tends to zero as $y \searrow 0$, so that
\formula{
  \frac{\partial}{\partial t_+} \expr{e^{\lambda t} P^D_t \psi_\lambda(x)} & = 0
}
for all $t > 0$ and $x > 0$. By~\eqref{eq:hl:genpsi:dom2} (with both sides multiplied by $e^{\lambda y}$), this also holds for $t = 0$. Finally, the function $e^{\lambda t} P^D_t \psi_\lambda(x)$ is continuous with respect to $t$ for each $x > 0$ (this follows from weak continuity of $p^D_t(x, z) dz$ with respect to $t$, which is a consequence of stochastic continuity of the killed Cauchy process; one can also prove this using the explicit formula for $p_t$ and~\eqref{eq:C:pDt}). It follows that $e^{\lambda t} P^D_t \psi_\lambda(x)$ is constant in $t \ge 0$, and since $P^D_0 \psi_\lambda(x) = \psi_\lambda(x)$, this completes the proof.
\end{proof}

\begin{remark}
Since $r(x) > 0$, the functions $\psi_\lambda$ are clearly not in $L^2(D)$, so the above result does not provide any information about the $L^2(D)$ properties of the operators $P^D_t$. This problem is studied in Section~\ref{sec:spec}.
\end{remark}

\begin{remark}
Also the derivatives $\psi_\lambda'$ are locally integrable eigenfunctions of $P^D_t$, continuous in $D$, but not at $0$. As this is not used in the sequel, we omit the proof.
\end{remark}

\begin{remark}
\label{rem:dilogarithm}
The functions $\psi_\lambda$ can be effectively computed by numerical integration. Indeed,~\eqref{eq:aux:int1},~\eqref{eq:aux:int2} and the identity
\formula{
  \int_0^t \frac{\log u}{1 + u^2} du & = \arctan t \log t + \frac{i}{2} (\mathrm{Li}_2(i t) - \mathrm{Li}_2(-i t)) ,
}
where $\mathrm{Li}_2(z)$ is the dilogarithm function, yield that
\formula{
  \psi_\lambda(x) & = \sin(\lambda x + \sfrac{\pi}{8}) - \frac{\sqrt{2}}{2 \pi} \int_0^\infty \frac{t^{1 + \frac{\arctan t}{\pi}}}{(1 + t^2)^{\frac{5}{4}}} \exp\expr{\frac{i}{2\pi} (\mathrm{Li}_2(i t) - \mathrm{Li}_2(-i t))} e^{-t \lambda x} dt .
}
\end{remark}

%
%

\section{Properties of the function $B$}
\label{sec:B}

\noindent
In this section we study the properties of the function $B$. As an interesting corollary, the Laplace transform of the eigenfunctions $\psi_\lambda$ is computed.

The function $B$ defined by~\eqref{eq:hl:h2} extends to a holomorphic function on $\C \setminus (-\infty, 0]$, satisfying $B(\bar{z}) = \overline{B(z)}$ (see also Appendix~B). Therefore $B$ is defined on whole $\C$, it is holomorphic in $\C \setminus (-\infty, 0]$ with a branch cut on $(-\infty, 0]$, and it is continuous on $\cl{\C}_+$. The following properties of $B$ will play an important role.

When $\imag z > 0$, we have
\formula{
  B(z) + B(-z) & = \frac{1}{\pi} \int_{-\infty}^0 \frac{\log(z - s)}{1 + s^2} \, ds + \frac{1}{\pi} \int_{-\infty}^0 \frac{\log(-z - s)}{1 + s^2} \, ds \\
  & = \frac{1}{\pi} \int_0^\infty \frac{\log(z + s)}{1 + s^2} \, ds + \frac{1}{\pi} \int_{-\infty}^0 \frac{\log(z + s) - i \pi}{1 + s^2} \, ds \\
  & = \frac{1}{\pi} \int_{-\infty}^\infty \frac{\log(z - s)}{1 + s^2} \, ds - \frac{i \pi}{2} \, .
}
On the right-hand side, the function $s \mapsto \log(z - s)$, holomorphic (and therefore harmonic) in $\C_-$, is integrated against the Poisson kernel of the lower half-plane $p_1(s) = \frac{1}{\pi} \frac{1}{1 + s^2}$. The result is the value of $\log(z - s)$ at $s = -i$. It follows that
\formula{
  B(z) + B(-z) & = \log(z - (-i)) - \frac{i \pi}{2} = \log(1 - i z) .
}
By $B(\bar{z}) = \cl{B(z)}$ we get $B(z) + B(-z) = \log(1 + i z)$ whenever $\imag z < 0$, and so
\formula[eq:hl:Bsym]{
  e^{B(z)} & = (1 - i z \sigma(z)) e^{-B(-z)} ,
}
where $\sigma(z) = 1$ when $\imag z > 0$ and $\sigma(z) = -1$ when $\imag z < 0$. A similar relation for $\eta$ was used earlier in~\eqref{eq:hl:F:1}, see also~\eqref{eq:aux:etasym}. By continuity of $B(z)$ in $\overline{\C}_+$ the formula~\eqref{eq:hl:Bsym} is also valid for $z \in \R$ if we let $\sigma(z) = 1$ for $z < 0$ and $\sigma(z) = -1$ for $z > 0$. For completeness, we let $\sigma(0) = 0$.

By~\eqref{eq:hl:F:1a}, \eqref{eq:hl:r}, the relation between $F$, $\psi_\lambda$ and $r_\lambda$, and using $B(t) = \eta(t) + i \arctan t_-$,
\formula[eq:hl:r1]{
  r(x) & = \frac{\sqrt{2}}{2 \pi} \int_0^\infty \tau(t) e^{-t x} dt , && \text{where} & \tau(t) & = \imag \frac{e^{B(-t)}}{1 + t^2} \, .
}
Note that by the definition, $\tau(t) = 0$ for $t \le 0$. In the sequel, we need the Hilbert transform of $\tau$, which can be computed as follows. The function $e^{B(z)} / (1 + z^2)$ is meromorphic in the upper half-plane with a simple pole at $i$, so that the function
\formula{
  \frac{e^{B(z)}}{1 + z^2} - \frac{1}{2} \frac{e^{B(i)}}{1 + i z} - \frac{1}{2} \frac{e^{\overline{B(i)}}}{1 - i z}
}
is holomorphic in $\C_+$. In fact it is in $H^p(\C_+)$ for $p \in (1, \infty)$, see~\eqref{eq:aux:Best}. Its boundary limit on $\R$ is equal to
\formula{
   \frac{e^{B(t)}}{1 + t^2} - \frac{1}{2} \frac{e^{B(i)}}{1 + i t} - \frac{1}{2} \frac{e^{\overline{B(i)}}}{1 - i t} & = \frac{e^{B(t)} - \sqrt{2} \cos \frac{\pi}{8} - t \sqrt{2} \sin \frac{\pi}{8}}{1 + t^2} ,
}
and the imaginary part of this function is just $\tau(-t)$. Therefore, the Hilbert transform of $\tau(-t)$ is the negative of the real part of the above function. It follows by~\eqref{eq:pre:hilbert:sym} that for $t \in \R$,
\formula[eq:hl:Htau]{
  H \tau(t) & = \real \frac{e^{B(-t)} - \sqrt{2} \cos \frac{\pi}{8} + t \sqrt{2} \sin \frac{\pi}{8}}{1 + t^2} \, .
}

We are now able to compute the Laplace transform $\lap \psi_\lambda$ of $\psi_\lambda$. By a direct computation, we have
\formula[eq:hl:L1]{
  \int_0^\infty \sin(x + \sfrac{\pi}{8}) e^{-t x} dx & = \frac{\cos \frac{\pi}{8} + t \sin \frac{\pi}{8}}{1 + t^2} \, , && t > 0 .
}
On the other hand, by Fubini's theorem and~\eqref{eq:hl:r1},
\formula{
  \int_0^\infty r(x) e^{-t x} dx & = \frac{\sqrt{2}}{2 \pi} \int_0^\infty \frac{\tau(s)}{t + s} \, ds = -\frac{\sqrt{2}}{2} H \tau(-t) , && t \ge 0 .
}
By~\eqref{eq:hl:Htau} we have
\formula[eq:hl:L2]{
  \int_0^\infty r(x) e^{-t x} dx & = - \frac{\sqrt{2}}{2} \frac{e^{B(t)}}{1 + t^2} + \frac{\cos \frac{\pi}{8} + t \sin \frac{\pi}{8}}{1 + t^2} \, , && t \ge 0 .
}
In particular,
\formula[eq:hl:Lr0]{
  \int_0^\infty r(x) dx & = \cos \frac{\pi}{8} - \frac{\sqrt{2}}{2} .
}
Formulas~\eqref{eq:hl:L1} and~\eqref{eq:hl:L2} give
\formula{
  \lap \psi_1(t) & = \int_0^\infty \psi_1(x) e^{-t x} dx = \frac{\sqrt{2}}{2} \frac{e^{B(t)}}{1 + t^2} \, , & t > 0 .
}
By scaling and the uniqueness of the holomorphic continuation, we obtain the following result.

\begin{corollary}
The Laplace transform of $\psi_\lambda$ is equal to
\formula[eq:hl:L]{
  \lap \psi_\lambda(z) & = \int_0^\infty \psi_\lambda(x) e^{-z x} dx = \frac{\sqrt{2}}{2} \frac{\lambda e^{B(\frac{z}{\lambda})}}{\lambda^2 + z^2} \, , & \real z > 0 ,
}
where $B(z)$ is given by~\eqref{eq:hl:h2}.
\end{corollary}

%
%

\section{The remainder term}
\label{sec:err}

\noindent
This section is devoted to a detailed analysis of the remainder term $r_\lambda$, see~\eqref{eq:hl:r}. Recall that $r_\lambda(x) = r(\lambda x)$, where
\formula[eq:err:r]{
  r(x) & = \frac{\sqrt{2}}{2 \pi} \int_0^\infty \frac{t}{(1 + t^2)^{\frac{5}{4}}} \exp\expr{\frac{1}{\pi} \int_0^t \frac{\log s}{1 + s^2} ds} e^{-t x} dt .
}
Since $r$ is the Laplace transform of a positive function, it is totally monotone, i.e. all functions $(-1)^n r^{(n)}$ are nonnegative and monotonically decreasing. In most of our estimates we simply use the inequality $-\mathcal{C} \le \int_0^t \frac{\log s}{1 + s^2} ds \le 0$ for $t > 0$ and formula~\eqref{eq:aux:beta} from Appendix~B. The $L^1(\R)$ norm of $r$, however, we already calculated in~\eqref{eq:hl:Lr0},
\formula[eq:err:rint]{
  \int_0^\infty r(x) dx & = \cos \frac{\pi}{8} - \frac{\sqrt{2}}{2} \in (0.216, 0.217) .
}
Since $\frac{1}{t+s} \le \frac{1}{2 \sqrt{t s}}$, by Fubini's theorem,
\formula[eq:err:rnorm]{
\begin{split}
  \int_0^\infty (r(x))^2 dx & \le \frac{1}{2 \pi^2} \int_0^\infty \int_0^\infty \frac{t}{(1 + t^2)^{\frac{5}{4}}} \frac{s}{(1 + s^2)^{\frac{5}{4}}} \frac{1}{t + s} dt ds \\
  & \le \frac{1}{4 \pi^2} \expr{\int_0^\infty \frac{\sqrt{t}}{(1 + t^2)^{\frac{5}{4}}} dt}^2 = \frac{(\Gamma(\frac{3}{4}))^2}{\pi (\Gamma(\frac{1}{4}))^2} < 0.037 .
\end{split}
}
In a similar manner, $\frac{1}{t+s} \ge \frac{1}{\sqrt{1 + t^2} \sqrt{1 + s^2}}$, so that
\formula[eq:err:rnorml]{
\begin{split}
  \int_0^\infty (r(x))^2 dx & \ge \frac{e^{-\frac{2 \mathcal{C}}{\pi}}}{2 \pi^2} \expr{\int_0^\infty \frac{t}{(1 + t^2)^{\frac{7}{4}}} dt}^2 = \frac{2 e^{-\frac{2 \mathcal{C}}{\pi}}}{9 \pi^2} > 0.012 .
\end{split}
}
For $x > 0$, we have
\formula[eq:err:r1]{
\begin{split}
  r(x) & = \frac{\sqrt{2}}{2 \pi} \int_0^\infty \frac{t}{(1 + t^2)^{\frac{5}{4}}} \exp\expr{\frac{1}{\pi} \int_0^t \frac{\log s}{1 + s^2} ds} e^{-t x} dt \\
  & \le \frac{\sqrt{2}}{2 \pi} \int_0^\infty t e^{-t x} dt \le \frac{\sqrt{2}}{2 \pi x^2} \, .
\end{split}
}
In a similar manner,
\formula[eq:err:rn1]{
  (-1)^n r^{(n)}(x) & \le \frac{\sqrt{2}}{2 \pi} \frac{(n+1)!}{x^{n+2}} \, .
}
Also,
\formula[eq:err:r2]{
  r(x) & \le r(0) = \sin \sfrac{\pi}{8} = \frac{\sqrt{2 - \sqrt{2}}}{2} < 0.383 ,
}
and
\formula[eq:err:rn2]{
  -r'(x) & \le \frac{\sqrt{2}}{2 \pi} \int_0^\infty \frac{e^{-t x}}{\sqrt{t}} dt = \frac{1}{\sqrt{2 \pi x}} \, .
}
It follows that for $x > 0$,
\formula{
  |\psi_1(x)| & \le \abs{\sin(x + \sfrac{\pi}{8}) - \sin \sfrac{\pi}{8}} + \abs{r(x) - r(0)} \\
  & \le x + \int_0^x |r'(y)| dy \le x + \sqrt{\frac{2 x}{\pi}} \, .
}
Since clearly $|\psi_1(x)| \le |\sin(x + \frac{\pi}{8})| + |r(x)| \le 2$, we have
\formula[eq:err:psisqrt]{
  |\psi_1(x)| & \le \min \expr{x + \sqrt{\frac{2 x}{\pi}}, 2} \le \min \expr{2 \sqrt{x}, 2} .
}
This property was already used in Section~\ref{sec:hl} in the proof of Theorem~\ref{th:halfline}.

Finally, since $r' < 0$, the first zero of $\psi_\lambda'$ is greater than $\frac{\pi}{8 \lambda}$. It follows that
\formula[eq:err:psi]{
  \norm{\psi_\lambda}_\infty & \le 1 + r_\lambda(\sfrac{\pi}{8 \lambda}) = 1 + r(\sfrac{\pi}{8}) \le 1 + \frac{\sqrt{2}}{2 \pi} \int_0^\infty \frac{t e^{-\frac{\pi}{8} t}}{(1 + t^2)^{\frac{5}{4}}} \, dt < 1.14 ;
}
for the last inequality, integrate by parts the left hand side of formula~3.387(7) in~\cite{bib:gr07}. The estimate~\eqref{eq:err:psi} is only used in Corollary~\ref{cor:bounded}, where a weaker version of~\eqref{eq:err:psi} would only result in a larger constant in~\eqref{eq:phi:five}. In fact, for the present constant $3$, we only need that $\norm{\psi_\lambda} \le 1.19$, which is easily obtained by estimating $(1 + t^2)^{-5/4}$ in the integrand in~\eqref{eq:err:psi} by a constant on each of the intervals $[\frac{k}{2}, \frac{k+1}{2}]$ with $k = 0, 1, ..., 7$, and $[4, \infty)$.

%
%

\section{Spectral representation of the transition semigroup for the half-line}
\label{sec:spec}

\noindent
Let $D = (0, \infty)$. In this section we study the $L^2(D)$ properties of the operators $P^D_t$. For $f \in C_c(D)$, define
\formula[eq:spec:pi]{
  \Pi f(x) & = \int_0^\infty f(\lambda) \psi_\lambda(x) d\lambda , && x \in D ,
}
where $\psi_\lambda = \sin(\lambda x + \frac{\pi}{8}) - r_\lambda(x)$ is given by~\eqref{eq:hl:psi} and~\eqref{eq:hl:r}. Note that
\formula{
  F_1(x) & = \int_0^\infty f(\lambda) \sin(\lambda x + \sfrac{\pi}{8}) d\lambda , && x \in D ,
}
satisfies $\norm{F_1}_2 \le c_1 \norm{f}_2$. Also, for
\formula{
  F_2(x) & = \int_0^\infty f(\lambda) r_\lambda(x) d\lambda , && x \in D ,
}
we may apply~\eqref{eq:err:r1} and~\eqref{eq:err:r2} to obtain
\formula{
  \int_0^\infty (F_2(x))^2 dx & \le \int_0^\infty \int_0^\infty \int_0^\infty |f(\mu)| |f(\lambda)| r(\mu x) r(\lambda x) d\mu d\lambda dx \\
  & \le \int_0^\infty \int_0^\infty \expr{\int_0^\infty \frac{c_2}{(1 + \mu x)^2 (1 + \lambda x)^2} dx} |f(\mu)| |f(\lambda)| d\mu d\lambda \\
  & \le \int_0^\infty \int_0^\infty \expr{\int_0^\infty \frac{c_2}{(1 + (\mu + \lambda) x)^2} dx} |f(\mu)| |f(\lambda)| d\mu d\lambda \\
  & = c_2 \int_0^\infty \int_0^\infty \frac{|f(\mu)| |f(\lambda)|}{\mu + \lambda} d\mu d\lambda ,
}
which is bounded by $c_2 \pi \norm{f}_2$ by Hardy-Hilbert's inequality. It follows that $\norm{\Pi f}_2 = \norm{F_1 + F_2}_2 \le c_3 \norm{f}_2$, and therefore $\Pi$ can be continuously extended to a unique bounded linear operator on $L^2(D)$.

For $f \in C_c(D)$, $p^D_t(x, y) f(\lambda) \psi_\lambda(y)$ is integrable in $(y, \lambda) \in D \times D$, so that by Theorem~\ref{th:halfline},
\formula[eq:spec:pt]{
  P^D_t \Pi f(x) & = \int_0^\infty e^{-\lambda t} f(\lambda) \psi_\lambda(x) d\lambda , && x \in D .
}
Let $f, g \in C_c(D)$ and define $f_k(\lambda) = e^{-k \lambda t} f(\lambda)$, $g_k(\lambda) = e^{-k \lambda t} g(\lambda)$. From~\eqref{eq:spec:pt} it follows that $P^D_t \Pi f_k = \Pi f_{k+1}$ and $P^D_t \Pi g_k = \Pi g_{k+1}$. Since the operators $P^D_t$ are self-adjoint, we have
\formula{
  \int_0^\infty \Pi f(x) \Pi g(x) dx & = \int_0^\infty P^D_t \Pi f_{-1}(x) \Pi g(x) dx \\
  & = \int_0^\infty \Pi f_{-1}(x) P^D_t \Pi g(x) dx = \int_0^\infty \Pi f_{-1}(x) \Pi g_1(x) dx .
}
By induction,
\formula{
  \int_0^\infty \Pi f(x) \Pi g(x) dx & = \int_0^\infty \Pi f_{-k}(x) \Pi g_k(x) dx .
}
Suppose that $\supp f \sub (0, \lambda_0)$ and $\supp g \sub (\lambda_0, \infty)$. Then we have
\formula{
  \int_0^\infty \Pi f(x) \Pi g(x) dx & = \int_0^\infty \Pi (e^{-k \lambda_0 t} f_{-k})(x) \Pi (e^{k \lambda_0 t} g_k)(x) dx .
}
Both $e^{-k \lambda_0 t} f_{-k}$ and $e^{k \lambda_0 t} g_k$ tend to zero uniformly as $k \rightarrow \infty$, and so $\Pi (e^{-k \lambda_0 t} f_{-k})$ and $\Pi (e^{k \lambda_0 t} g_k)$ converge to zero in $L^2(D)$. We conclude that $\Pi f$ and $\Pi g$ are orthogonal in $L^2(D)$. By an approximation argument, this is true for any $f, g \in L^2(D)$, provided that $f(\lambda) = 0$ for $\lambda \ge \lambda_0$ and $g(\lambda) = 0$ for $\lambda \le \lambda_0$.

Define
\formula{
  \mu(A) & = \int_0^\infty (\Pi \ind_A(x))^2 dx , && A \sub D .
}
Clearly
\formula{
  \mu(A) & \le c_3 \norm{\ind_A}_2^2 = c_3 |A| , && A \sub D.
}
Whenever $A \sub (0, \lambda_0)$ and $B \sub (\lambda_0, \infty)$, we have
\formula{
  \mu(A \cup B) & = \int_0^\infty (\Pi \ind_A(x))^2 dx + \int_0^\infty (\Pi \ind_B(x))^2 dx + 2 \int_0^\infty \Pi \ind_A(x) \Pi \ind_B(x) dx \\
  & = \mu(A) + \mu(B) .
}
Finally, when $A = \bigcup_{n = 1}^\infty A_n$, where $A_1 \sub A_2 \sub ...$ and $|A| < \infty$, the sequence $\ind_{A_n}$ converges in $L^2(D)$ to $\ind_A$ as $n \rightarrow \infty$. Hence $\Pi \ind_{A_n}$ converges to $\Pi \ind_A$ in $L^2(D)$, and so $\mu(A) = \lim_{n \rightarrow \infty} \mu(A_n)$. It follows that $\mu$ is an absolutely continuous measure on $(0, \infty)$. By an approximation argument, we have
\formula{
  \int_0^\infty \Pi f(x) \Pi g(x) dx & = \int_0^\infty f(\lambda) g(\lambda) \mu(d\lambda)
}
for any $f, g \in L^2(D)$.

Note that $\psi_\lambda(q x) = \psi_{\lambda/q}(x)$, and therefore $\Pi f_q(x) = q \Pi f(q x)$, where $f_q(x) = f(\frac{x}{q})$. It follows that $\mu(q A) = q \mu(A)$ and so $\mu$ must be a multiple of the Lebesgue measure on $(0, \infty)$, say $\mu(A) = c_4 |A|$. This result is a version of Plancherel's theorem, where Fourier transform is replaced by $\Pi$:
\formula{
  \int_0^\infty \Pi f(x) \Pi g(x) dx & = c_4 \int_0^\infty f(\lambda) g(\lambda) d\lambda
}
for any $f, g \in L^2(D)$.

The constant $c_4$ can be determined by considering $f(\lambda) = \frac{1}{\sqrt{q}} \ind_{[1, 1+q]}(\lambda)$, $q > 0$. We then have $\norm{f}_2 = 1$. On the other hand,
\formula{
  \Pi f(x) & = \frac{1}{x \sqrt{q}} \expr{\cos(x + \sfrac{\pi}{8}) - \cos((1+q) x + \sfrac{\pi}{8})} - \frac{1}{\sqrt{q}} \int_1^{1+q} r(\lambda x) d\lambda .
}
The $L^2(D)$ norm of the first summand converges to $\sqrt{\frac{\pi}{2}}$ as $q \searrow 0$, just as in the case of the Fourier sine transform. The second summand is bounded by $\sqrt{q} r(x)$ and so it converges to zero in $L^2(D)$. It follows that $c_4 = \frac{\pi}{2}$. The Plancherel's theorem can be therefore written as
\formula[eq:spec:plancherel]{
  \int_0^\infty \Pi f(x) \Pi g(x) dx & = \frac{\pi}{2} \int_0^\infty f(\lambda) g(\lambda) d\lambda
}
In particular, $\sqrt{\frac{2}{\pi}} \, \Pi$ is an isometry on $L^2(D)$. Since $\psi_\lambda(x) = \psi_x(\lambda)$, for $f, g \in C_c(D)$ (and therefore for any $f, g \in L^2(D)$) we also have
\formula{
  \int_0^\infty \Pi f(x) g(x) dx & = \int_0^\infty f(\lambda) \Pi g(\lambda) d\lambda ,
}
which combined with~\eqref{eq:spec:plancherel} yields that $\Pi^2 f = \frac{\pi}{2} f$. We collect the above results in the following theorem.

\begin{theorem}
\label{th:spec}
The operator $\sqrt{\frac{2}{\pi}} \, \Pi : L^2(D) \rightarrow L^2(D)$ gives a spectral representation of $\A_D$ and the semigroup $(P^D_t)$, acting on $L^2(D)$, where $D = (0, \infty)$; that is, for any $f \in L^2(D)$,
\begin{enumerate}
\item[(a)] $\norm{f}_2 = \sqrt{\frac{2}{\pi}} \norm{\Pi f}$ (Plancherel's theorem);
\item[(b)] $\Pi P^D_t f(\lambda) = e^{-\lambda t} \Pi f(\lambda)$;
\item[(c)] $f$ is in the domain of $\A_D$ if and only if $\lambda \Pi f(\lambda)$ is square integrable;
\item[(d)] $\Pi \A_D f(\lambda) = -\lambda \Pi f(\lambda)$.
\end{enumerate}
Furthermore, $\Pi^2 = \frac{\pi}{2} \id$ (inversion formula).
\end{theorem}

%
%

\section{Transition density for the half-line}
\label{sec:pdt}

\noindent
The aim of this section is to compute an explicit formula of the transition density $p^D_t(x, y)$ of the Cauchy process killed on exiting a half-line $D = (0,\infty)$, or the heat kernel for $\A_D$. Let us note that the transition density of the Brownian motion killed on exiting a half-line $(0,\infty)$ equals $\frac{1}{\sqrt{2 \pi t}} e^{-\frac{|x - y|^2}{2 t}} - \frac{1}{\sqrt{2 \pi t}} e^{-\frac{|x + y|^2}{2 t}}$, which follows from the reflection principle. For the Cauchy process we cannot use the reflection principle and the computation of $p^D_t(x, y)$ requires using much more complicated methods.

\begin{theorem}
\label{th:heat}
For $D = (0, \infty)$ and any $g \in L^p$, $p \in [1,\infty]$, we have
\formula[eq:spec:heat]{
  P^D_t g(x) & = \int_0^\infty p^D_t(x, y) g(y) dy, \quad \quad t,x > 0,
}
where
\formula[eq:spec:heatkernel]{
  p^D_t(x, y) & = \frac{1}{\pi} \frac{t}{t^2 + (x-y)^2} - \frac{1}{x y} \int_0^t \frac{f(\frac{s}{x}) f(\frac{t-s}{y})}{\frac{s}{x} + \frac{t-s}{y}} \, ds, \quad \quad t,x,y > 0,
}
and
\formula[eq:spec:deff]{
  f(s) & = \frac{1}{\pi} \frac{s}{1 + s^2} \exp\expr{\frac{1}{\pi} \int_0^\infty \frac{\log(s + w)}{1 + w^2} \, dw}, \quad \quad s > 0.
}
\end{theorem}

\begin{remark}
\label{rem:heat}
Note that for $s > 0$, $f$ is positive continuous and bounded. This follows by the fact that $f(s)  = \frac{1}{\pi} \frac{s}{1 + s^2} e^{\eta(s)}$ and~\eqref{eq:aux:est}. The function $p^D_t(x, y)$ can be effectively computed by numerical integration. Indeed, by the same arguments as in Remark~\ref{rem:dilogarithm} we have
\formula{
  f(s) & = \frac{1}{\pi} \frac{s^{1 - \frac{\arctan(s)}{\pi}}}{(1 + s^2)^{\frac{3}{4}}} \exp\expr{\frac{-i}{2 \pi} (\mathrm{Li}_2(i s) - \mathrm{Li}_2(-i s))},
}
where $\mathrm{Li}_2$ is the dilogarithm function.
\end{remark}

\begin{proof}[Proof of Theorem~\ref{th:heat}]
For $g \in C_c(D)$ we have $\Pi P^D_t g(\lambda) = e^{-\lambda t} \Pi g(\lambda)$ (see~\eqref{eq:spec:pi} and Theorem~\ref{th:spec}). Applying $\Pi^{-1} = \frac{2}{\pi} \Pi$ to both sides of this identity yields
\formula{
  P^D_t g(x) & = \frac{2}{\pi} \int_0^\infty e^{-\lambda t} \Pi g(\lambda) \psi_\lambda(x) d\lambda = \frac{2}{\pi} \int_0^\infty \int_0^\infty e^{-\lambda t} \psi_\lambda(x) \psi_\lambda(y) g(y) dy d\lambda .
}
By the Fubini's theorem,~\eqref{eq:spec:heat} holds with
\formula[eq:spec:kernel]{
  p^D_t(x, y) & = \frac{2}{\pi} \int_0^\infty \psi_\lambda(x) \psi_\lambda(y) e^{-\lambda t} d\lambda .
}
By an approximation argument,~\eqref{eq:spec:heat} holds for $g \in L^p(\R)$ with any $p \in [1, \infty]$. We will now prove~\eqref{eq:spec:heatkernel}.

Suppose first that $x < y$, and let $t = t_1 + t_2$, $t_1, t_2 > 0$. By Plancherel's theorem and identities $\psi_\lambda(x) = \psi_x(\lambda)$, $\overline{\lap \psi_y(z)} = \lap \psi_y(\bar{z})$, we have
\formula[eq:spec:R]{
\begin{split}
  p^D_t(x, y) & =\frac{2}{\pi} \int_0^\infty (\psi_x(\lambda) e^{-t_1 \lambda}) (\psi_y(\lambda) e^{-t_2 \lambda}) d\lambda \\
  & = \frac{1}{\pi^2} \int_{-\infty}^\infty \lap \psi_x(t_1 + i s) \lap \psi_y(t_2 - i s) \, ds = \frac{1}{2 \pi i} \int_{t_1 - i \infty}^{t_1 + i \infty} R(z) dz ,
\end{split}
}
where (see~\eqref{eq:hl:L})
\formula{
  R(z) & = \frac{2}{\pi} \lap \psi_x(z) \lap \psi_y(t - z) = \frac{1}{\pi} \frac{x y \exp(B(\frac{z}{x}) + B(\frac{t - z}{y}))}{(x^2 + z^2)(y^2 + (t - z)^2)} .
}
Note that $R$ is defined on $\C$ and it is meromorphic in $\C \setminus ((-\infty, 0] \cup [t, \infty))$ with simple poles at $\pm i x$ and $t \pm i y$. Let $z \in \C \setminus [0,t]$. By~\eqref{eq:hl:Bsym} and the identity $(1 + i z \sigma(z))(1 - i z \sigma(z)) = 1 + (z \sigma(z))^2 = 1 + z^2$, we have for all $z \in \C$,
\formula{
  R(z) & = \frac{1}{\pi} \frac{\bigl( 1 - i \frac{z}{x} \sigma(\frac{z}{x}) \bigr) \bigl( 1 - i \frac{t - z}{y} \sigma(\frac{t - z}{y}) \bigr) \exp(-B(-\frac{z}{x}) - B (-\frac{t - z}{y}))}{x y \bigl( 1 + \frac{z^2}{x^2} \bigr) \bigl( 1 + \frac{(t - z)^2}{y^2} \bigr) } \\
  & = \frac{1}{\pi} \frac{\exp(-B(-\frac{z}{x}) - B(-\frac{t - z}{y}))}{x y \bigl( 1 + i \frac{z}{x} \sigma(\frac{z}{x}) \bigr) \bigl( 1 + i \frac{t - z}{y} \sigma(\frac{t - z}{y}) \bigr) } \, .
}
Since $\sigma(\frac{t - z}{y}) = -\sigma(\frac{z}{x})$ for $z \in \C \setminus [0,t]$, it follows that for $z \in \C \setminus [0, t]$,
\formula{
  R(z) & = \frac{\exp(-B(-\frac{z}{x}) - B(-\frac{t - z}{y}))}{\pi x y \bigl( \frac{z}{x} + \frac{t - z}{y} \bigr) } \expr{\frac{\frac{z}{x}}{1 + i \frac{z}{x} \sigma(\frac{z}{x})} + \frac{\frac{t - z}{y}}{1 + i \frac{t - z}{y} \sigma(\frac{t - z}{y})} } .
}
We therefore have $R(z) = R_1(z) + R_2(z)$ for $z \in \C \setminus [0,t]$, where, again using~\eqref{eq:hl:Bsym},
\formula[eq:spec:R1]{
\begin{split}
  R_1(z) & = \frac{\exp(-B(-\frac{z}{x}) - B(-\frac{t - z}{y}))}{\pi x y \bigl( \frac{z}{x} + \frac{t - z}{y} \bigr) } \cdot \frac{\frac{z}{x}}{1 + i \frac{z}{x} \sigma(\frac{z}{x})} \\
  & = \frac{\exp(B(\frac{z}{x}) - B(-\frac{t - z}{y}))}{\pi x y \bigl( \frac{z}{x} + \frac{t - z}{y} \bigr) } \cdot \frac{\frac{z}{x}}{1 + \frac{z^2}{x^2}} \, ,
\end{split}
}
and
\formula[eq:spec:R1a]{
  R_1(z) & = \frac{\exp(B(\frac{z}{x}) + B(\frac{t - z}{y}))}{\pi x y \bigl( \frac{z}{x} + \frac{t - z}{y} \bigr) } \cdot \frac{\frac{z}{x}}{1 + \frac{z^2}{x^2}} \cdot \frac{1}{1 - i \frac{t - z}{y} \sigma(\frac{t - z}{y})} \, .
}
Also, in a similar manner,
\formula[eq:spec:R2]{
  R_2(z) & = \frac{\exp(-B(-\frac{z}{x}) + B(\frac{t - z}{y}))}{\pi x y \bigl( \frac{z}{x} + \frac{t - z}{y} \bigr) } \cdot \frac{\frac{t - z}{y}}{1 + \frac{(t - z)^2}{y^2}} \, ,
}
and
\formula[eq:spec:R2a]{
  R_2(z) & = \frac{\exp(B(\frac{z}{x}) + B(\frac{t - z}{y}))}{\pi x y \bigl( \frac{z}{x} + \frac{t - z}{y} \bigr) } \cdot \frac{\frac{t - z}{y}}{1 + \frac{(t - z)^2}{y^2}} \cdot \frac{1}{1 - i \frac{z}{x} \sigma(\frac{z}{x})} \, .
}
The only zero of $\frac{z}{x} + \frac{t - z}{y}$ is $z = \frac{t x}{x - y} < 0$. Hence $R_1(z)$ is holomorphic in the set $\set{\real z > 0} \setminus [0, t]$ (by~\eqref{eq:spec:R1}), bounded in the neighborhood of $[0, t]$, and it decays as $|z|^{-2}$ at infinity (by~\eqref{eq:aux:Bintest}). Also, $R_2(z)$ is meromorphic in the set $\set{\real z < t} \setminus [0, t]$ (by~\eqref{eq:spec:R2}) with a simple pole at $\frac{t x}{x - y}$, bounded near $[0, t]$, and it decays as $|z|^{-2}$ at infinity.

For $n = 1, 2, ...$ let $\gamma$ be the positively oriented contour consisting of:
\begin{itemize}
  \item two vertical segments $\gamma_1 = \left[t_1 - n i, t_1 - \frac{i}{n}\right]$, $\gamma_5 = \left[t_1 + \frac{i}{n}, t_1 + n i\right]$,
  \item two horizontal segments $\gamma_2 = \left[t_1 - \frac{i}{n}, -\frac{i}{n}\right]$, $\gamma_4 = \left[\frac{i}{n}, t_1 + \frac{i}{n}\right]$, 
  \item two semi-cirles $\gamma_3 = \set{|z| = \frac{1}{n}, \, \real z \le 0}$ and $\gamma_6 = \set{|z - t_1| = n, \, \real z \le t_1}$.
\end{itemize}
Clearly, $\int_{\gamma_1 \cup \gamma_5} R_2(z) dz$ converges to $\int_{t_1 - i \infty}^{t_1 + i \infty} R_2(z) dz$ as $n \to \infty$. The integrals over $\gamma_3$ and $\gamma_6$ converge to zero by the properties of $R_2$. Finally, by~\eqref{eq:spec:R2a},
\formula{
  \int_{\gamma_2 \cup \gamma_4} R_2(z) dz & \to \int_0^{t_1} \frac{\exp(B(\frac{s}{x}) + B(\frac{t - s}{y}))}{\pi x y \bigl( \frac{s}{x} + \frac{t - s}{y} \bigr) } \cdot \frac{\frac{t - s}{y}}{1 + \frac{(t - s)^2}{y^2}} \expr{\frac{1}{1 - i \frac{s}{x}} - \frac{1}{1 + i \frac{s}{x}}} ds \\
  & = \int_0^{t_1} \frac{2 \pi i f(\frac{s}{x}) f(\frac{t - s}{y})}{x y \bigl( \frac{s}{x} + \frac{t - s}{y} \bigr) } ds .
}
Therefore, by the residue theorem,
\formula[eq:spec:intR2]{
  \frac{1}{2 \pi i} \int_{t_1 - i \infty}^{t_1 + i \infty} R_2(z) dz & = -\int_0^{t_1} \frac{f(\frac{s}{x}) f(\frac{t - s}{y})}{x y \bigl( \frac{s}{x} + \frac{t - s}{y} \bigr) } ds + \res(R_2, \sfrac{t x}{x - y}) .
}
In a similar manner, using~\eqref{eq:spec:R1a} and analogous contours $\gamma$ consisting of two segments of the line $\real z = t_1$, two segments parallel to $[t_1, t]$, and two semi-circles centered at $t$ (the small one) and $t_1$ (the large one), both contained in $\set{\real z \ge t_1}$, we obtain that
\formula[eq:spec:intR1]{
  \frac{1}{2 \pi i} \int_{t_1 - i \infty}^{t_1 + i \infty} R_1(z) dz & = -\int_{t_1}^t \frac{f(\frac{s}{x}) f(\frac{t - s}{y})}{x y \bigl( \frac{s}{x} + \frac{t - s}{y} \bigr) } ds .
}
Therefore,~\eqref{eq:spec:R}, \eqref{eq:spec:intR2} and~\eqref{eq:spec:intR1} yield that
\formula{
  p^D_t(x, y) & = - \frac{1}{x y} \int_0^t \frac{f(\frac{s}{x}) f(\frac{t - s}{y})}{\frac{s}{x} + \frac{t - s}{y}} \, ds + \res(R_2, \sfrac{t x}{x - y}) \, .
}
For $z = \frac{t x}{x - y}$ we have $\frac{z}{x} = \frac{t}{x - y} = -\frac{t - z}{y}$. Therefore, by~\eqref{eq:spec:R2} we get
\formula{
  \res(R_2, \sfrac{t x}{x - y}) & = \frac{1}{\pi (y - x)} \cdot \frac{-\frac{t}{x - y}}{1 + \frac{t^2}{(x - y)^2}} = \frac{1}{\pi} \frac{t}{t^2 + (x - y)^2} \, ,
}
and~\eqref{eq:spec:heatkernel} follows for $x < y$.

When $x > y$, simply note that $p^D_t(x, y) = p^D_t(y, x)$ (see~\eqref{eq:spec:kernel} or e.g.~\cite{bib:cs97}, Theorem 2.4), and that the right-hand side of~\eqref{eq:spec:heatkernel} has the same symmetry property (this follows by a substitution $s = t - v$). Finally, for $x = y$ simply use the continuity of $p^D_t(x, y)$ and $f$.
\end{proof}

For the next result, we need the following simple observation, similar to the derivation of~\eqref{eq:hl:Htau}. By~\eqref{eq:hl:Bsym} we have $\imag e^{-B(-s)} = -\frac{s}{1+s^2} e^{B(s)}$ for $s > 0$. Hence the function $f$ defined by~\eqref{eq:spec:deff} satisfies 
\formula{
  f(s) & = \frac{1}{\pi} \frac{s}{1 + s^2} e^{\eta(s)} = 
  \frac{1}{\pi} \frac{s}{1 + s^2} e^{B(s)} = -\frac{1}{\pi} \imag e^{-B(-s)}, && s > 0 .
}
If we extend $f$ by $f(s) = 0$ for $s < 0$, then $f(-s) = \frac{1}{\pi} \imag e^{-B(s)}$ for all real $s$. Since $e^{-B(z)}$ is in $H^p(\C_+)$ for $p \in (2, \infty)$ (see~\eqref{eq:aux:Best}), the Hilbert transform of $f$ is given by (see~\eqref{eq:pre:hilbert:sym})
\formula{
  H f(s) & = - \frac{1}{\pi} \real e^{-B(-s)} = - \frac{1}{\pi} e^{-\eta(-s)} , && s \in \R .
}
It follows that 
\formula[eq:spec:Hfa]{
  H f(-s) & = - \frac{1}{\pi^2} \frac{s}{1 + s^2} \frac{1}{f(s)} \, , && s > 0, && \text{and} & H f(0) & = -\frac{1}{\pi} \, .
}

\begin{theorem}
\label{th:time}
For $D = (0, \infty)$, we have
\formula[eq:spec:time]{
  \pr^x(\tau_D \in dt) & = \frac{1}{\pi} \frac{x}{t^2 + x^2} \exp\expr{\frac{1}{\pi} \int_0^\infty \frac{\log(\frac{t}{x} + w)}{1 + w^2} \, dw} dt .
}
\end{theorem}

Using the function $f$ defined in~\eqref{eq:spec:deff}, we have $\pr^x(\tau_D \in dt) = \frac{1}{t} f(\frac{t}{x}) dt$.

\begin{proof}
By Theorem~\ref{th:heat} we have
\formula[eq:spec:timedistribution]{
\begin{split}
  \pr^x(\tau_D > t) & = \int_0^\infty p^D_t(x, y) dy \\
  & = \frac{1}{\pi} \int_0^{\infty} \frac{t}{t^2 + (x-y)^2} \, dy - \int_0^t \int_0^\infty \frac{1}{x y}  \frac{f(\frac{s}{x}) f(\frac{t-s}{y})}{\frac{s}{x} + \frac{t-s}{y}} \, dy ds.
\end{split}
}
By a substitution $w = (t - s)/y$ we obtain
\formula{
  \int_0^\infty \frac{1}{x y} \frac{f(\frac{s}{x}) f(\frac{t-s}{y})}{\frac{s}{x} + \frac{t-s}{y}} \, dy & = \frac{f(\frac{s}{x})}{x} \int_0^\infty \frac{f(w)}{w (\frac{s}{x} + w)} \, dw \\
  & = \frac{f(\frac{s}{x})}{s} \expr{\int_0^\infty \frac{f(w)}{w} \, dw - \int_0^\infty \frac{f(w)}{\frac{s}{x} + w} \, dw} . 
}
The right-hand side equals $\frac{\pi}{s} f(\frac{s}{x}) (-H f(0) + H f(-\sfrac{s}{x}))$. This,~\eqref{eq:spec:Hfa} and~\eqref{eq:spec:timedistribution} give 
\formula{
  \pr^x(\tau_D > t) & = \frac{1}{\pi} \int_0^\infty \frac{t}{t^2 + (x-y)^2} \, dy - \int_0^t \frac{f(\frac{s}{x})}{s} \, ds + \frac{1}{\pi} \int_0^t \frac{x}{x^2 + s^2} \, ds.
}
By substitution $v = x - y$ in the first integral and $v = x t / s$ in the third one,
\formula{
  \pr^x(\tau_D > t) & = \frac{1}{\pi} \int_{-\infty}^x \frac{t}{t^2 + v^2} \, dv - \int_0^t \frac{f(\frac{s}{x})}{s} \, ds + \frac{1}{\pi} \int_x^\infty \frac{t}{t^2 + v^2} \, dv \\
  & = 1 - \int_0^t \frac{f(\frac{s}{x})}{s} \, ds.
}
The result follows by differentiation and~\eqref{eq:spec:deff}.
\end{proof}

\begin{remark}
Theorem~\ref{th:time} can also be obtained in a more explicit manner. In fact, by scaling properties of $X_t$, we have $\pr_x(\tau_D > t) = g(\frac{x}{t})$ for some function $g$ continuous in $\R$, vanishing on $(-\infty, 0]$. Furthermore, $\pr_x(\tau_D > t)$ satisfies the heat equation in $D$, i.e. $\frac{\partial}{\partial t} \, \pr_x(\tau_D > t) = -\sqrt{-\frac{d^2}{d x^2}} \, \pr_x(\tau_D > t)$. For $t = 1$ this gives $-x g'(x) = -\sqrt{-\frac{d^2}{d x^2}} \, g(x)$, $x > 0$. Since $-\sqrt{-\frac{d^2}{d x^2}} \, g = H g'$, we have $H g'(s) = -s g'(s)$ for $s > 0$.

Let $h(s) = \frac{1}{s} g'(-\frac{1}{s})$, $h(0) = 0$. Then it can be shown that $h$ is continuous on $\R$, and by the definition~\eqref{eq:pre:H} of the Hilbert transform, $H h(s) = -\frac{1}{s} H g'(-\frac{1}{s})$. It follows that $h(s) = s H h(s)$ for $s < 0$, and $h(s) = 0$ for $s \ge 0$. Therefore $H h - i h$ is a boundary limit of some holomorphic function in $\C_+$, and $(H h(s) - i h(s)) e^{-i \arctan s_-}$ is real for all $s \in \R$. This problem can be solved using the method applied in Section~\ref{sec:hl}, and the solution is $H h(s) - i h(s) = c e^{-B(s)}$ with some $c \in \R$. Therefore $g'(s) = -\frac{1}{s} h(-\frac{1}{s}) = \frac{c}{s} \imag e^{-B(1/s)}$ for $s > 0$, and finally
\formula{
  \pr_x(\tau_D \in dt) & = -\sfrac{x}{t^2} g'(\sfrac{x}{t}) = -\sfrac{c}{t} \imag e^{-B(t/x)} ,
}
which agrees with~\eqref{eq:spec:time} when $c = \frac{1}{\pi}$. The details of this alternative argument are left to the interested reader.
\end{remark}

\begin{remark}
\label{rem:pdt}
The integral of $p^D_t(x, y)$ with respect to $t \in (0, \infty)$ is the Green function of $(X_t)$ on the half-line, given by the well-known explicit formula of M.~Riesz, see e.g.~\cite{bib:bgr61}. Also, the distribution of $X(\tau_D)$ (and even the joint distribution of $\tau_D$ and $X(\tau_D)$) is determined by $p^D_t(x, y)$, see~\cite{bib:iw62}. Explicit formulas for Green functions and exit distributions for some related processes in half-lines and intervals were found recently in~\cite{bib:bmr09,bib:bmr:pre}.
\end{remark}

Theorem~\ref{th:time} implies a new result for the $2$-dimensional Brownian motion. Namely we obtain the distribution of some local time of the $2$-dimensional Brownian motion killed at some entrance time. For the $1$-dimensional Brownian motion similar results were widely studied and are usually called Ray-Knight theorems~\cite{bib:k63,bib:r63,bib:ry99}.

\begin{corollary}
\label{cor:Brownian}
Let $B_t = (B^{(1)}_t, B^{(2)}_t)$ be the $2$-dimensional Brownian motion and $L(t) = \lim_{\eps \to 0^+} \frac{1}{2 \eps} \int_0^t \chi_{(-\eps,\eps)}(B^{(2)}_s) \, ds$ be the local time of $B_t$ on the line $(-\infty,\infty) \times \set{0}$. Let $A = (-\infty,0] \times \set{0}$ and $T_A = \inf \set{t \ge 0 \; : \; B(t) \in A}$ be the first entrance time for $A$. Then for any $x > 0$ we have
\formula{
  P^{(x,0)}(L(T_A) \in dt) & = \frac{1}{\pi} \frac{x}{x^2 + t^2} \exp\left(\frac{1}{\pi} \int_{0}^{\infty} \frac{\log\left(\frac{t}{x} +s\right)}{1 + s^2} \, ds\right) \, dt.
}
For $(x,y) \in \R^2$, $y \ne 0$ and $t \ge 0$ we have
\formula{
  P^{(x,y)}(L(T_A) \le t) & = \frac{1}{\pi} \int_{-\infty}^{0} \frac{|y|}{y^2 + (x - u)^2} \, du \\
  & \hspace{20mm} + \frac{1}{\pi} \int_{0}^{\infty} \frac{|y|}{y^2 + (x - u)^2} P^{(u,0)}(L(T_A) \le t) \, du.
}
\end{corollary}

\begin{proof}
Let $\eta_t = \inf\{s > 0: \, L(s) > t\}$ be the inverse of the local time $L(t)$. It is well known (see e.g.~\cite{bib:s58}) that the $1$-dimensional Cauchy process $X_t$ can be identified with $B^{(1)}(\eta_t)$. With this relation, we have $L(T_A) = \tau_{(0,\infty)}$, where $\tau_{(0,\infty)} = \inf \set{t \ge 0 \; : \; X_t \notin (0,\infty)}$. This and Theorem~\ref{th:time} give the first equality. The second equality follows by the harmonicity of $(x,y) \to P^{(x,y)}(L(T_A) \le t)$ in $\{(x,y) \in \R^2: \, y > 0 \}$ and in $\{(x,y) \in \R^2: \, y < 0 \}$.
\end{proof}

%
%

\section{Approximation to eigenfunctions on the interval}
\label{sec:int}

\noindent
In this section the interval $D = (-1, 1)$ is studied. Let $n$ be a positive integer and $\mu_n = \frac{n \pi}{2} - \frac{\pi}{8}$. Our goal is to show that $\mu_n$ is close to $\lambda_n$, the $n$-th eigenvalue of the semigroup $(P^D_t)$.

Let $q$ be the function equal to $0$ on $(\infty, -\frac{1}{3})$ and to $1$ on $(\frac{1}{3}, \infty)$, defined by~\eqref{eq:aux:q} in Appendix~C. We construct approximations to eigenfunctions of $(P^D_t)$ by combining the eigenfunctions $\psi_{\mu_n}(1 + x)$ and $\psi_{\mu_n}(1 - x)$ for half-line, studied in Section~\ref{sec:hl}. For a symmetric eigenfunction, when $n$ is odd, let
\formula[eq:int:phisym]{
\begin{split}
  \tilde{\ph}_n(x) & = q(-x) \psi_{\mu_n}(1 + x) + q(x) \psi_{\mu_n}(1 - x) \\
  & = (-1)^{\frac{n-1}{2}} \cos(\mu_n x) \ind_D(x) - q(-x) r_{\mu_n}(1 + x) - q(x) r_{\mu_n}(1 - x) .
\end{split}
}
For an antisymmetric eigenfunction, when $n$ is even, we define
\formula[eq:int:phiasym]{
\begin{split}
  \tilde{\ph}_n(x) & = q(-x) \psi_{\mu_n}(1 + x) - q(x) \psi_{\mu_n}(1 - x) \\
  & = (-1)^{\frac{n}{2}} \sin(\mu_n x) \ind_D(x) - q(-x) r_{\mu_n}(1 + x) + q(x) r_{\mu_n}(1 - x) .
\end{split}
}

\begin{lemma}
With the above definitions,
\formula[eq:int:approx]{
  \norm{\A_D \tilde{\ph}_n + \mu_n \tilde{\ph}_n}_2 & < \sqrt{1.21 + \frac{8.00}{\mu_n} + \frac{13.66}{\mu_n^2}} \cdot \frac{1}{\mu_n} .
}
\end{lemma}

\begin{proof}
Note that we have
\formula{
  \tilde{\ph}_n(x) & - \psi_{\mu_n}(1 + x) = -(1 - q(-x)) \psi_{\mu_n}(1 + x) - (-1)^n q(x) \psi_{\mu_n}(1 - x) \\
  & = -q(x) (\psi_{\mu_n}(1 + x) + (-1)^n \psi_{\mu_n}(1 - x)) \\
  & = q(x) (r_{\mu_n}(1 + x) + (-1)^n r_{\mu_n}(1 - x)) - \sin({\mu_n} (1 + x) + \sfrac{\pi}{8}) \ind_{[1, \infty)}(x) .
}
Denote $h(x) = \sin(\mu_n (1 + x) + \frac{\pi}{8}) \ind_{[1, \infty)}(x)$ and $f(x) = r_{{\mu_n}}(1 + x) + (-1)^n r_{\mu_n}(1 - x)$, $g(x) = q(x) f(x)$. By~\eqref{eq:err:r1},~\eqref{eq:err:rn1} and~\eqref{eq:err:rint},
\formula{
  M_0 & = \sup_{x \in (-\frac{1}{3}, \frac{1}{3})} |f(x)| \le r(\sfrac{2 \mu_n}{3}) + r(\sfrac{4 \mu_n}{3}) \le \frac{45 \sqrt{2}}{32 \pi \mu_n^2} , \\
  M_1 & = \sup_{x \in (-\frac{1}{3}, \frac{1}{3})} |f'(x)| \le -\mu_n r'(\sfrac{2 \mu_n}{3}) - \mu_n r'(\sfrac{4 \mu_n}{3}) \le \frac{243 \sqrt{2}}{64 \pi \mu_n^2} , \\
  M_2 & = \sup_{x \in (-\frac{1}{3}, \frac{1}{3})} |f''(x)| \le \mu_n^2 r''(\sfrac{2 \mu_n}{3}) + \mu_n^2 r''(\sfrac{4 \mu_n}{3}) \le \frac{4131 \sqrt{2}}{256 \pi \mu_n^2} , \\
  I & = \int_0^\infty |f(x)| dx \le \int_0^\infty r_{\mu_n}(1 + x) dx + \int_0^1 r_{\mu_n}(1 - x) dx \\
  & = \frac{1}{\mu_n} \int_0^\infty r(y) dy = \expr{\cos \frac{\pi}{8} - \frac{\sqrt{2}}{2}} \frac{1}{\mu_n} \, ;
}
notation here corresponds to that of Appendix~C. By~\eqref{eq:aux:genest1} and~\eqref{eq:aux:genest2},
\formula[]{
  \label{eq:int:est1} |\A_D g(z)| & < \frac{0.605}{\mu_n^2} + \frac{0.156}{\mu_n} , && z \in (-1, -\sfrac{1}{3}) ; \\
  \label{eq:int:est2} |\A_D g(z)| & < \frac{4.444}{\mu_n^2} + \frac{0.622}{\mu_n} , && z \in (-\sfrac{1}{3}, 0) .
}
Furthermore, $|g(z)| = 0$ for $z \in (-1, -\frac{1}{3})$ and
\formula[eq:int:est3] {
  |\mu_n g(z)| & \le \sfrac{\mu_n}{2} M_0 < \frac{0.317}{\mu_n} \, , && z \in (-\sfrac{1}{3}, 0) .
}
Finally, for $z < 0$ we have
\formula{
  |(-\Delta)^\frac{1}{2} h(z)| & = \frac{1}{\pi} \abs{\int_1^\infty \frac{\sin(\mu_n (1 + x) + \frac{\pi}{8})}{(x - z)^2} \, dx} \\
  & \le \frac{1}{\pi (1 - z)^2} \int_1^{1 + \frac{\pi}{\mu_n}} \abs{\sin(\mu_n (1 + x) + \sfrac{\pi}{8})} \, dx = \frac{1}{\pi \mu_n (1 - z)^2} \, ,
}
so that
\formula[]{
  \label{eq:int:est4} |(-\Delta)^\frac{1}{2} h(z)| & < \frac{0.180}{\mu_n} \, , && z \in (-1, -\sfrac{1}{3}) ; \\
  \label{eq:int:est5} |(-\Delta)^\frac{1}{2} h(z)| & < \frac{0.319}{\mu_n} \, , && z \in (-\sfrac{1}{3}, 0) .
}
Since for $z \in (-1, 0)$ we have
\formula{
  |\A_D \tilde{\ph}_n(z) + \mu_n \tilde{\ph}_n(z)| & \le |(-\Delta)^{\frac{1}{2}} h(z)| + |(-\Delta)^{\frac{1}{2}} g(z)| + |\mu_n g(z)| ,
}
estimates~\eqref{eq:int:est1}--\eqref{eq:int:est5} yield that
\formula[]{
  \label{eq:int:est6} |\A_D \tilde{\ph}_n(z) + \mu_n \tilde{\ph}_n(z)| & < \frac{0.605}{\mu_n^2} + \frac{0.336}{\mu_n} \, , && z \in (-1, -\sfrac{1}{3}) ; \\
  \label{eq:int:est7} |\A_D \tilde{\ph}_n(z) + \mu_n \tilde{\ph}_n(z)| & < \frac{4.444}{\mu_n^2} + \frac{1.258}{\mu_n} \, , && z \in (-\sfrac{1}{3}, 0) .
}
By symmetry, estimates similar to~\eqref{eq:int:est6} and~\eqref{eq:int:est7} hold for $z \in (0, 1)$. The estimate~\eqref{eq:int:approx} follows.
\end{proof}

The estimate of the $L^2(D)$ norm of $\tilde{\ph}_n$ plays an important role in the sequel. We have
\formula[eq:int:phinorm]{
  \sqrt{1 - \frac{0.52}{\mu_n}} \le \norm{\tilde{\ph}_n}_2 & \le \sqrt{1 + \frac{1.37}{\mu_n}} \, .
}
Indeed, the lower bound follows by~\eqref{eq:err:rint}, \eqref{eq:int:phisym}, \eqref{eq:int:phiasym} and symmetry,
\formula{
  \norm{\tilde{\ph}_n}_2^2 & \ge \int_{-1}^1 \expr{\sin(\mu_n (x + 1) + \sfrac{\pi}{8})}^2 dx \\
  & \qquad - 4 \int_{-1}^1 \abs{q(-x) r_{\mu_n}(x + 1) \sin(\mu_n (x + 1) + \sfrac{\pi}{8})} dx \\
  & \ge \expr{1 + \frac{\sqrt{2}}{4 \mu_n}} - \frac{4}{\mu_n} \expr{\cos \frac{\pi}{8} - \frac{\sqrt{2}}{2}} .
}
In a similar manner, using also~\eqref{eq:err:rnorm},
\formula{
  \norm{\tilde{\ph}_n}_2^2 & \le \expr{1 + \frac{\sqrt{2}}{4 \mu_n}} + \frac{4}{\mu_n} \expr{\cos \frac{\pi}{8} - \frac{\sqrt{2}}{2}} + 4 \int_{-1}^1 (r(\mu_n (x + 1)))^2 dx \\
  & \le \expr{1 + \frac{\sqrt{2}}{4 \mu_n}} + \frac{4}{\mu_n} \expr{\cos \frac{\pi}{8} - \frac{\sqrt{2}}{2}} + \frac{4 (\Gamma(\frac{3}{4}))^2}{\pi (\Gamma(\frac{1}{4}))^2 \mu_n} \, .
}

%
%

\section{Simplicity of eigenvalues for the interval}
\label{sec:eigv}

\noindent
We continue denoting by $\ph_j$ the eigenfunctions of $(P^D_t)$, by $\lambda_j$ ($\lambda_j > 0$) the corresponding eigenvalues, and by $\tilde{\ph}_n$ and $\mu_n$ the approximations of the previous section. Fix $n \ge 1$. Since $\tilde{\ph}_n \in L^2(D)$, we have $\tilde{\ph}_n = \sum_j a_j \ph_j$ for some $a_j$. Moreover, $\norm{\tilde{\ph}_n}_2^2 = \sum_j a_j^2$ and $\A_D \tilde{\ph}_n = -\sum_j \lambda_j a_j \ph_j$. Let $\lambda_{k(n)}$ be the eigenvalue nearest to $\mu_n$. Then
\formula{
  \norm{\A_D \tilde{\ph}_n + \mu_n \tilde{\ph}_n}_2^2 & = \sum_{j = 1}^\infty (\lambda_j - \mu_n)^2 a_j^2 \\
  & \ge (\lambda_{k(n)} - \mu_n)^2 \sum_{j = 1}^\infty a_j^2 \ge (\lambda_{k(n)} - \mu_n)^2 \norm{\tilde{\ph}_n}_2^2 .
}
By~\eqref{eq:int:approx} and~\eqref{eq:int:phinorm}, it follows that
\formula[eq:int:lambda]{
  \abs{\lambda_{k(n)} - \mu_n} & \le \sqrt{\frac{1.21 + \frac{8.00}{\mu_n} + \frac{13.66}{\mu_n^2}}{1 - \frac{0.52}{\mu_n}}} \cdot \frac{1}{\mu_n} \, .
}
The right-hand side is a decreasing function of $n$, so that $\abs{\lambda_{k(n)} - \mu_n} < 0.098 \pi < \frac{\pi}{10}$ whenever $n \ge 4$. Hence we have the following result.

\begin{lemma}
\label{lem:int:eig}
Each interval $(\frac{n \pi}{2} - \frac{\pi}{4}, \frac{n \pi}{2})$, $n \ge 4$, contains an eigenvalue $\lambda_{k(n)}$.
\end{lemma}

In particular $\lambda_{k(n)}$ are distinct for $n \ge 4$. We will now prove that there are only three eigenvalues not included in the above lemma. For $t > 0$, we have (see e.g.~\cite{bib:bk09,bib:k98})
\formula{
  \sum_{j = 1}^\infty e^{-\lambda_j t} & = \int_D \sum_{j = 1}^\infty e^{-\lambda_j t} (\ph_j(x))^2 dx = \int_D p^D_t(x, x) dx \le \int_D p_t(0) dx = \frac{2}{\pi t} \, .
}
On the other hand,
\formula{
  \sum_{n = 4}^\infty e^{-\lambda_{k(n)} t} & \ge \sum_{n = 4}^\infty e^{-\frac{n \pi}{2} t} = \frac{e^{-2 \pi t}}{1 - e^{-\frac{\pi}{2} t}} \ge \frac{2}{\pi t} - \frac{7}{2}
}
for small $t > 0$. It follows that there are at most $3$ eigenvalues of $(P^D_t)$ other than $\lambda_{k(n)}$ ($n \ge 4$). Furthermore, we have $1 < \lambda_1 < \frac{3 \pi}{8}$, $2 \le \lambda_2 \le \pi$ and $3.4 \le \lambda_3 \le \frac{3 \pi}{2}$ by~\cite{bib:bk04}. Therefore, $k(n) = n$ for $n \ge 4$, and also by~\eqref{eq:int:lambda}, $\lambda_3 > 3.83$. We have thus proved the following theorem.

\begin{theorem}
\label{th:int:lambda}
We have
\formula{
  1 < \lambda_1 & < \frac{3 \pi}{8} , & 2 \le \lambda_2 & \le \pi , & 3.83 < \lambda_3 & \le \frac{3 \pi}{2} ,
}
and
\formula{
  \frac{n \pi}{2} - \frac{\pi}{8} - \frac{\pi}{10} < \lambda_n & < \frac{n \pi}{2} - \frac{\pi}{8} + \frac{\pi}{10} && (n \ge 4) .
}
In particular, all eigenvalues of $(P^D_t)$ are simple, $|\lambda_n - \lambda_m| > 0.69$ when $n \ne m$, and $|\lambda_n - \lambda_m| > \frac{3\pi}{10}$ if moreover $n \ge 4$. Furthermore, as $n \rightarrow \infty$,
\formula[eq:int:asymp]{
  \lambda_n = \frac{n \pi}{2} - \frac{\pi}{8} + O\expr{\frac{1}{n}} .
}
\end{theorem}

More precisely,
\formula[eq:int:asymp2]{
  \abs{\lambda_n - \expr{\frac{n \pi}{2} - \frac{\pi}{8}}} \le \frac{1}{n} \, && n \ge 1 ,
}
i.e. the constant in $O(\frac{1}{n})$ notation in~\eqref{eq:int:asymp} is not greater than $1$. Indeed, by~\eqref{eq:int:lambda}, formula~\eqref{eq:int:asymp2} holds for $n \ge 7$, and for $n \le 6$ one can use the estimates~\eqref{eq:num}. Without referring to numerical calculation of upper and lower bounds, one can use~\eqref{eq:int:lambda} for $n \ge 4$ and estimates of $\lambda_1$, $\lambda_2$ and $\lambda_4$ of Theorem~\ref{th:int:lambda} to obtain~\eqref{eq:int:asymp2} with $\frac{1}{n}$ replaced by $\frac{3}{2 n}$.

Better numerical bounds for first few eigenvalues are obtained in Section~\ref{sec:num}.

%
%

\section{Estimates of eigenfunctions for the interval}
\label{sec:eigf}

\noindent
In the preceding two sections the approximations $\tilde{\ph}_n$ to the eigenfunctions $\ph_n$ were constructed and it was proved that $\mu_n = \frac{n \pi}{2} - \frac{\pi}{8}$ is close to $\lambda_n$. Now we show that $\tilde{\ph}_n$ is close to $\ph_n$ in $L^2(D)$.

Let $n \ge 4$ be fixed. Recall that $\tilde{\ph}_n = \sum_j a_j \ph_j$; with no loss of generality we may assume that $a_n > 0$.  For $j \ne n$ we have $|\mu_n - \lambda_j| \ge \frac{3 \pi}{10}$. Therefore,
\formula{
  \norm{\A_D \tilde{\ph}_n + \mu_n \tilde{\ph}_n}_2^2 & = \sum_{j = 1}^\infty (\mu_n - \lambda_j)^2 a_j^2 \ge (\mu_n - \lambda_n)^2 a_n^2 + \frac{9 \pi^2}{100} \sum_{j \ne n} a_j^2 .
}
Denote the left-hand side by $M_n^2$; the upper bound for $M_n$ is given in~\eqref{eq:int:approx}. We have
\formula{
  \norm{\tilde{\ph}_n - a_n \ph_n}_2^2 & = \sum_{j \ne n} a_j^2 \le \frac{100 M_n^2}{9 \pi^2} \, .
}
Therefore,
\formula{
  \bigl\| \tilde{\ph}_n - \norm{\tilde{\ph}_n}_2 \ph_n \bigr\|_2 & \le \norm{\tilde{\ph}_n - a_n \ph_n} + (\norm{\tilde{\ph}_n}_2 - a_n) \le 2 \norm{\tilde{\ph}_n - a_n \ph_n} \le \frac{20 M_n}{3 \pi} .
}
This, together with~\eqref{eq:int:phinorm}, yields the following result.

\begin{lemma}
\label{lem:int:phi}
Let $n \ge 4$. With the notation of the previous two sections, we have $1 - \frac{0.52}{\mu_n} < \norm{\tilde{\ph}_n}^2 < 1 + \frac{1.37}{\mu_n}$, and
\formula{
  \bigl\| \tilde{\ph}_n - \norm{\tilde{\ph}_n}_2 \ph_n \bigr\|_2 & \le \frac{20}{3 \pi} \sqrt{1.21 + \frac{8.00}{\mu_n} + \frac{13.66}{\mu_n^2}} \cdot \frac{1}{\mu_n} .
}
\end{lemma}

In particular, for $n \ge 4$, by the above result and~\eqref{eq:int:phinorm},
\formula{
  \norm{\frac{\tilde{\ph}_n}{\norm{\tilde{\ph}_n}_2} - \ph_n}_2 & < \frac{20}{3 \pi} \cdot \frac{M_n}{\sqrt{1 - \frac{0.52}{\mu_n}}} < \frac{20}{3 \pi} \cdot \frac{\pi}{10} = \frac{2}{3} .
}
Since $\tilde{\ph}_n$ is symmetric or antisymmetric when $n$ is odd or even respectively, we have the alternating type of symmetry of $\ph_n$.

\begin{corollary}
\label{cor:sign}
The function $\ph_n$ is symmetric when $n$ is odd, and antisymmetric when $n$ is even.
\end{corollary}

\begin{proof}
For $n \le 3$ this is a result of~\cite{bib:bk04}. When $n \ge 4$, $\ph_n$ is either symmetric or antisymmetric, and the distance between $\ph_n$ and normed $\tilde{\ph}_n$ does not exceed $\frac{2}{3}$. Therefore $\ph_n$ has the same type of symmetry as $\tilde{\ph}_n$.
\end{proof}

\begin{corollary}
\label{cor:L2approx}
As $n \rightarrow \infty$,
\formula{
  \norm{\ph_n - \sin \expr{(\sfrac{n \pi}{2} - \sfrac{\pi}{8}) (1 + x) + \sfrac{\pi}{8}}}_2 & = O \expr{\frac{1}{\sqrt{n}}} .
}
\end{corollary}

By a rather standard argument, $\norm{\ph_n}_\infty \le \sqrt{\frac{e \lambda_n}{\pi}}$, see e.g.~\cite{bib:k08}. A slight modification gives the following result.

\begin{proposition}
Let $c = \norm{\tilde{\ph}_n}_2$. Then
\formula[eq:phi:uniform]{
  \norm{\ph_n}_\infty \le \frac{1}{c} \expr{\sqrt{\frac{e \lambda_n}{\pi}} \cdot \norm{c \ph_n - \tilde{\ph}_n} + \sqrt{e} \norm{\psi_{\mu_n}}_\infty} .
}
\end{proposition}

\begin{proof}
Let $t = \frac{1}{2 \lambda_n}$. Using Cauchy-Schwarz inequality, Plancherel theorem and the inequality $p^D_t(x, y) \le p_t(x - y)$, we obtain
\formula{
  c \abs{\ph_n(x)} & \le e^{\lambda_n t} \abs{P^D_t(c \ph_n - \tilde{\ph}_n)(x)} + e^{\lambda_n t} \abs{P^D_t \tilde{\ph}_n(x)} \\
  & \le \sqrt{e} \cdot \sqrt{\int_{-\infty}^\infty (p_t(x - y))^2 dy} \cdot \norm{c \ph_n - \tilde{\ph}_n}_2 + \sqrt{e} \norm{\tilde{\ph}_n}_\infty \\
  & = \sqrt{\frac{e}{2 \pi} \int_{-\infty}^\infty e^{-2 t \abs{z}} dz} \cdot \norm{c \ph_n - \tilde{\ph}_n} + \sqrt{e} \norm{\psi_{\mu_n}}_\infty \\
  & = \sqrt{\frac{e}{2 \pi t}} \cdot \norm{c \ph_n - \tilde{\ph}_n} + \sqrt{e} \norm{\psi_{\mu_n}}_\infty ,
}
and the proposition follows.
\end{proof}

\begin{corollary}
\label{cor:bounded}
The functions $\ph_n(x)$ are uniformly bounded in $n \ge 1$ and $x \in D$. 
\end{corollary}

More precisely, for $n \ge 1$ we have
\formula[eq:phi:five]{
  \norm{\ph_n}_\infty & \le 3 .
}
Indeed, for $n \ge 7$ this follows from~\eqref{eq:phi:uniform} when the right-hand side is estimated using Theorem~\ref{th:int:lambda}, Lemma~\ref{lem:int:phi} and~\eqref{eq:err:psi}. For $n \le 6$ it is a consequence of $\norm{\ph_n}_\infty \le \sqrt{\frac{e \lambda_n}{\pi}}$ and $\lambda_n \le \frac{n \pi}{2}$.

%
%

\section{Numerical estimates}
\label{sec:num}

\noindent
In this section we give numerical estimates for the eigenvalues $\lambda_n$ of the semigroup $(P^D_t)$ when $D = (-1, 1)$. The following estimates hold true; the upper bounds are given in superscript and the lower bounds in subscript:
\formula[eq:num]{
\begin{array}{r@{\,}l@{\hspace{2cm}}r@{\,}l}
  \lambda_1 & = 1.157773883697^{92}_{58}  & \lambda_6    & = \hspace*{5pt}9.032852690^{50838}_{48857} \\
  \lambda_2 & = 2.75475474221^{695}_{510}   & \lambda_7    & = 10.6022930996^{3854}_{1113} \\
  \lambda_3 & = 4.31680106659^{758}_{303}   & \lambda_8    & = 12.1741182627^{6180}_{2585} \\
  \lambda_4 & = 5.8921474709^{4751}_{3908}  & \lambda_9    & = 13.744109059^{44402}_{39799} \\
  \lambda_5 & = 7.460175739^{41122}_{39764} & \lambda_{10} & = 15.3155549960^{8382}_{2690}
\end{array}
}
This is the result of numerical computation of the eigenvalues of $900 \times 900$ matrices using \emph{Mathematica 6.01}. Different methods are used for the upper and lower bounds, as is described below. For the introduction to the notions of the Green operator and the Green function, the reader is referred to e.g.~\cite{bib:bg68}. The explicit formula for the Green function of the interval was first obtained by M.~Riesz~\cite{bib:r38}.

\subsection{Upper bounds}
\label{sec:num:upper}

For the upper bounds, we use the Rayleigh-Ritz method, see e.g.~\cite{bib:v73}. Let $G_D$ be the Green operator for $P^D_t$. Then $G_D \ph_n = \frac{1}{\lambda_n} \ph_n$. The following min-max variational characterization of eigenvalues of $G_D$ is well known, see e.g.~\cite{bib:rs78}:
\formula[eq:wariat]{
  \frac{1}{\lambda_n} & = \max \set{\min_{f \in E} R(f) \; : \; \text{$E$ is $n$-dimensional subspace of $L^2(D)$}} ,
}
where $R(f)$ is the Rayleigh quotient for $G_D$,
\formula{
  R(f) & = \frac{\int_{-1}^1 f(x) G_D f(x) dx}{\norm{f}_2^2} .
}
Let $f_n$, $n = 1, 2, ...$, be a complete orthonormal system in $L^2(D)$, and let $E_N$ be the subspace spanned by $f_n$, $n = 1, 2, ..., N$. By replacing $L^2(D)$ by $E_N$ in~\eqref{eq:wariat}, we clearly obtain the upper bound $\lambda_{n,N}^+$ for $\lambda_n$, $n = 1, 2, ..., N$. On the other hand, $(\lambda_{n, N}^+)^{-1}$ is the $n$-th largest eigenvalue of the $N \times N$ matrix $A_N$ of coefficients $a_{m,n}$ of the operator $G_D$ in the basis $(f_1, f_2, ..., f_N)$ (note that $a_{m,n}$ do not depend on $N$).

The main difficulty is to find a convenient basis $f_n$ for which the approximations converge sufficiently fast, while the entries of $A_N$ can be computed explicitly.

For the sake of comparison, recall that analytical computation in \cite{bib:bk04} gives the upper bound $\frac{3 \pi}{8} \approx 1.178$. Our first attempt to use the Rayleigh-Ritz method for $\A_D$ instead of $G_D$, with $f_n(x) = \sin(\frac{n \pi}{2} (x + 1))$, resulted in relatively poor estimates. For example, for $N = 1000$ the upper bound for the first eigenvalue is $\lambda_{1,1000} \approx 1.1579$, accurate up to third decimal place. A more efficient approach, described below, uses Legendre polynomials.

We begin with computation the values of the Green operator of the interval $(-1,1)$ on the polynomials $g_n(x) = x^n$. Recall that the Green function of the interval $D = (-1,1)$ for the Cauchy process is given by
\formula{
  G_D(x,y) & = \frac{1}{2\pi}\int_0^{\frac{(1-x^2)(1-y^2)}{(x-y)^2}}\frac{du}{\sqrt{u}\sqrt{u+1}}
  = \frac{1}{\pi} \log\frac{1-xy+\sqrt{1-x^2}\sqrt{1-y^2}}{|x-y|}\/,
}
where $x,y \in D$. Integrating by parts gives, after some simplification,
\formula{
  G_D g_n(y) &= \int_{-1}^1 G_D(x,y) g_n(x) dx \\
  &= \frac{1}{\pi}\frac{\sqrt{1-y^2}}{n+1} \pv\int_{-1}^1 \frac{x^{n+1}\,dx}{\sqrt{1-x^2}(x-y)}\\
  &= \frac{1}{\pi}\frac{\sqrt{1-y^2}}{n+1} \int_{-1}^1 \frac{(x^{n+1}-y^{n+1})\,dx}{\sqrt{1-x^2}(x-y)}+\frac{1}{\pi}\frac{\sqrt{1-y^2}y^{n+1}}{n+1}I(y) ,
}   
where
\formula{
I(y) = \pv\int_{-1}^1 \frac{dx}{\sqrt{1-x^2}(x-y)}\/.
}   
The indefinite integral is given by
\formula{
  \frac{1}{\sqrt{1-y^2}} \log \frac{|x-y|\sqrt{1-y^2}}{2(1-xy+\sqrt{1-x^2}\sqrt{1-y^2})} ,
}
and therefore $I(y) = 0$. Consequently, we have
\formula{
  G_D g_n(y) & = \frac{1}{\pi} \frac{\sqrt{1 - y^2}}{n + 1} \int_{-1}^1 \frac{(x^{n+1} - y^{n+1}) \, dx}{\sqrt{1 - x^2} (x - y)} \\
  & = \frac{1}{\pi} \frac{\sqrt{1 - y^2}}{n + 1} \sum_{i = 0}^n y^{n-i} \int_{-1}^1 \frac{x^i dx}{\sqrt{1 - x^2}} \\
  & = \frac{1}{\sqrt{\pi}} \frac{\sqrt{1 - y^2}}{n + 1} \sum_{j = 0}^{\floor{\frac{n}{2}}} y^{n - 2j} \frac{\Gamma(j + \frac{1}{2})}{\Gamma(j + 1)} \/ .
}  
Finally, for $m, n = 0, 1, 2, ...$ such that $m + n$ is even we get
\formula{
  G_{m,n} &= \int_{-1}^1 g_m(y) G_D g_n(y) dy \\
  & = \frac{1}{n + 1} \sum_{j = 0}^{\floor{\frac{n}{2}}} \frac{\Gamma(j + \frac{1}{2})}{\sqrt{\pi} \, \Gamma(j + 1)} \int_{-1}^1 \sqrt{1 - y^2} \, y^{n+m-2j} dy \\
  & = \frac{1}{2(n + 1)} \sum_{j = 0}^{\floor{\frac{n}{2}}} \frac{\Gamma(j + \frac{1}{2})}{\Gamma(j + 1)} \frac{\Gamma(\frac{n + m + 1}{2} - j)}{\Gamma(\frac{n + m}{2} + 2 - j)} \/ .
}
By simple induction, one can prove that in this case
\formula[eq:num:gmn]{
  G_{m,n} & = \begin{cases}
    \dfrac{1}{m+n+2} \cdot \dfrac{\Gamma(\frac{m+1}{2})}{\Gamma(\frac{m}{2} + 1)} \cdot  \dfrac{\Gamma(\frac{n+1}{2})}{\Gamma(\frac{n}{2}+1)} & \text{for $m$, $n$ even} , \vspace*{5pt} \\
    \dfrac{1}{m+n+2} \cdot \dfrac{\Gamma(\frac{m}{2} + 1)}{\Gamma(\frac{m + 3}{2})} \cdot \dfrac{\Gamma(\frac{n}{2} + 1)}{\Gamma(\frac{n + 3}{2})} & \text{for $m$, $n$ odd} .
  \end{cases}
}
If $m + n$ is odd, we obviously have $G_{m,n} = 0$. 

The Legendre polynomials are defined by
\formula{
  f_n(x) & = \sum_{i=0}^{\floor{\frac{n}{2}}} c_{n,i} x^{n-2i} , 
}
where 
\formula[eq:num:cin]{
  c_{n,i} &= \frac{(-1)^i (2n-2i)!}{2^n i!(n-i)!(n-2i)!} = \frac{(-1)^i \Gamma(2n-2i+1)}{2^n i!\Gamma(n-i+1)\Gamma(n-2i+1)},
}
form the orthogonal basis in $L^2(D)$. Therefore, we have
\formula[eq:num:upper]{
  a_{m,n} & = \int_{-1}^1 f_m(y) G_D f_n(y) dy
  = \sum_{i = 0}^{\floor{\frac{m}{2}}} \sum_{j = 0}^{\floor{\frac{n}{2}}} c_{n,i} c_{m,j} G_{i,j} ,
}
with $c_{n,i}$ and $G_{i,j}$ given by~\eqref{eq:num:gmn} and~\eqref{eq:num:cin}. The upper bound for $\lambda_n$ is $\lambda_{n,N}^+$, where $(\lambda_{n,N}^+)^{-1}$ is the $n$-th greatest eigenvalue of the $N \times N$ matrix $A_N = (a_{n,m})$.

\subsection{Lower bounds}
\label{sec:num:lower}

To find the lower bounds to the eigenvalues of the problem (\ref{eq:intro:1})-(\ref{eq:intro:3}) for an interval $D=(-1,1)$ we apply the Weinstein-Aronszajn method of intermediate problems. More precisely, we use (with small changes in the notation) the method described in~\cite{bib:fk83} (section \emph{The method}), where the sloshing problem is considered. For more details, see~\cite{bib:fk83} and the references therein.

The analytic function $\sin(z) = (\sin \xi \cosh \eta, \sinh \xi \cos \eta)$, where $z = \xi + i \eta$, transforms the semi-infinite strip $R= \{(\xi,\eta) \in \R^2 \, : \, -\frac{\pi}{2} \le \xi \le \frac{\pi}{2}, \, \eta \ge 0\}$ onto the upper half-space $H = \{(x, y) \in \R^2 \, : \, y \ge 0\}$. Let $u$ be a solution to the eigenproblem \eqref{eq:intro:1}-\eqref{eq:intro:3} with $D = (-1, 1)$. Then the image $v(z) = u(\eta(z))$ of the function $u$ under $\eta$ is a solution to the following equivalent problem 
\formula[]{
  \label{eq:num:low:equiv:1} & \Delta v(\xi, \eta) = 0 && -\sfrac{\pi}{2}< \xi < \sfrac{\pi}{2}  , \, \eta > 0 \/, \\
  \label{eq:num:low:equiv:2} & \sfrac{\partial}{\partial \eta} v(\xi, 0) = -\lambda \cos \xi\, v(\xi, 0) && \sfrac{\pi}{2}\leq \xi \leq \sfrac{\pi}{2} , \, \eta=0 \/,\\
  \label{eq:num:low:equiv:3} & v(-\sfrac{\pi}{2}, \eta) = v(\sfrac{\pi}{2}, \eta) = 0 && \eta\geq 0  \/.
}
For $f \in L^2(-\frac{\pi}{2}, \frac{\pi}{2})$ we denote by $A f$ (not to be confused with $\A f$) the normal derivative of the harmonic function agreeing with $f$ on $(-\frac{\pi}{2}, \frac{\pi}{2})$ and vanishing on $\{-\frac{\pi}{2}, \frac{\pi}{2} \} \times [0,\infty)$ (this is an analogue of the Dirichlet-Neumann operator). Since $v(\xi, \eta) = \sin(k (\xi + \frac{\pi}{2})) e^{-k \eta}$ satisfies~\eqref{eq:num:low:equiv:1} and~\eqref{eq:num:low:equiv:3}, the eigenfunctions of $A$ are simply $g_k(\xi) = \sqrt{\frac{2}{\pi}} \sin(k (\xi + \frac{\pi}{2}))$, and $A g_k = k g_k$.

We define the operator of multiplication by the function $\sign \xi \sqrt{1 - \cos \xi}$
\formula{
  (T f)(\xi) & = \sign \xi \sqrt{1-\cos \xi} \, f(\xi), && f\in L^2(-\sfrac{\pi}{2}, \sfrac{\pi}{2}) .
}
The problem~\eqref{eq:num:low:equiv:1}--\eqref{eq:num:low:equiv:3} can be written in the operator form as
\formula[eq:num:low:OperatorForm]{
  (A f)(\xi) & = \lambda (1-T^2) f(\xi) .
}
Let $P_N$ be the orthogonal projection of $L^2(D)$ onto a linear subspace $E_N$ of $L^2(D)$ spanned by the first $N$ of the linearly dense set of functions $f_1, f_2, ...$. Then the eigenvalues $\lambda_{n,N}^-$ of the spectral problem 
\formula[Truncated_A]{
  A f & = \lambda (1 - T P_N T) f
}
are lower bounds for the eigenvalues of~\eqref{eq:num:low:OperatorForm} and consequently to the eigenvalues $\lambda_n$ of the problem~\eqref{eq:num:low:equiv:1}--\eqref{eq:num:low:equiv:3}. Roughly, this is because
\formula{
  \int_{-\frac{\pi}{2}}^{\frac{\pi}{2}} f(x) T P_N T f(x) dx & = \norm{P_N T f(x)}_2^2 \le \norm{T f(x)}_2^2 = \int_{-\frac{\pi}{2}}^{\frac{\pi}{2}} f(x) T^2 f(x) dx ,
}
and so the Rayleigh quotient associated with~\eqref{Truncated_A} is dominated by the Rayleigh quotient for~\eqref{eq:num:low:OperatorForm}, namely
\formula{
  \frac{\int_{-\frac{\pi}{2}}^{\frac{\pi}{2}} f(x) A f(x) dx}{\int_{-\frac{\pi}{2}}^{\frac{\pi}{2}} f(x) (1 - T P_N T) f(x) dx} & \le \frac{\int_{-\frac{\pi}{2}}^{\frac{\pi}{2}} f(x) A f(x) dx}{\int_{-\frac{\pi}{2}}^{\frac{\pi}{2}} f(x) (1 - T^2) f(x) dx} \, .
}
The problem~\eqref{Truncated_A} is called the intermediate problem. We will later choose $f_n$ so that each $T f_n$ is a linear combination of $g_i$, the eigenfunctions of $A$, say
\formula[beta_definition]{
  T f_n & = \sum_{i=1}^K c_{n,i} g_i, & n = 1, 2, ..., N ,
}
where $K \ge N$. Let $C$ be the $N \times K$ matrix with entries $c_{n,i}$, and let $B$ be the $N \times N$ Gram matrix of the functions $f_1, ..., f_N$, i.e. the matrix with entries
\formula{
  b_{m, n} = \int_{-\frac{\pi}{2}}^{\frac{\pi}{2}} f_m(x) f_n(x) dx .
}
Finally, let $D$ be the $K \times K$ diagonal matrix of the first $K$ eigenvalues $1, 2, ..., K$ of $A$. Note that for each $j > K$, the function $g_j$ is the solution of~\eqref{Truncated_A} with eigenvalue $\lambda = j$ (this is because $T g_j = 0$). On the other hand, if $f$ is the linear combination of $g_1, g_2, ..., g_K$ with coefficients $\alpha = (\alpha_1, ..., \alpha_K)$, then $f$ satisfies~\eqref{Truncated_A} if and only if $\alpha$ is the solution to the $K \times K$ relative matrix eigenvalue problem,
\formula[eq:num:matrix]{
  D \alpha & = \lambda (I - C^T B^{-1} C) \alpha .
}
By arranging the eigenvalues of~\eqref{eq:num:matrix} and eigenvalues $K+1, K+2, ...$ in the nondecreasing order, we obtain the sequence of eigenvalues $\lambda_{n,N}^-$ of the intermediate problem~\eqref{Truncated_A}. As it was already noted, these are lower bounds for $\lambda_n$.

We define 
\formula{
  f_n(x) & = 2 \sqrt{1+\cos x} \, g_{n}(x).
}
It follows that
\formula{
   T f_n(x) & = 2 \sin x \, g_n(x) = (-1)^n g_{n-1}(x)+ (-1)^{n+1} g_{n+1}(x) ,
}
using the convention that $g_0(x) = 0$. Consequently, $C$ is $N \times (N+1)$ matrix of the form
\formula{
  C & = \begin{pmatrix}
    0 & 1 & 0 & 0 & \cdots &0&0&0\\
    1 & 0 & -1 & 0 & \cdots &0&0&0\\ 
    0 & -1& 0 & 1 & \cdots & 0&0&0\\
    0 & 0& 1 & 0 & \cdots & 0&0&0\\
    \vdots &\vdots &\vdots &\vdots &\ddots & \vdots& \vdots& \vdots\\
    0 & 0& 0 & 0 & \cdots & 0 & (-1)^N & 0 \\
    0 & 0& 0 & 0 & \cdots & (-1)^N & 0 & (-1)^{N+1}
    \end{pmatrix} .
}
The coefficients of the Gram matrix $B$ can be easily computed, and we have
\formula{
  b_{m,n} & = \dfrac{(-1)^{1 + \frac{m+n}{2}} \, 32 m n}{\pi ((m-n)^2 - 1)((m+n)^2 - 1)} + 4\delta_{m,n}
}
whenever $m + n$ is even, and $b_{m,n} = 0$ otherwise. Finally, the solutions of the spectral problem~\eqref{eq:num:matrix} are simply the inverses of the eigenvalues of the matrix $D^{-1} (I - C^T B^{-1} C)$. These numbers turn out to be less than $N + 2$, therefore they form $\lambda_{n,N}^-$, $n = 1, 2, ..., N+1$.

%
%

\appendix

%
%

\section{Estimates of $p_t - p^D_t$}

\noindent
Let $D = (0, \infty)$. Let $p_t(x, A) = \pr_x(X_t \in A)$ for $A \sub \R$, and fix $x > 0$. By the strong Markov property,
\formula{
  2 \pr_x(X_t \le 0) & = 2 \pr_x(X_t \in D^c) = \ex_x(2 p_{t - \tau_D}(X(\tau_D), D^c) \; ; \; \tau_D \le t) .
}
Since $2 p_s(y, D^c) \ge 1$ for $y \le 0$, $s > 0$, the right-hand side is bounded below by $\pr_x(\tau_D \le t)$. Therefore, for $t > 0$ and $x > 0$,
\formula[eq:C:tauD]{
  \pr_x(\tau_D \le t) & \le \frac{2}{\pi} \int_{-\infty}^0 \frac{t}{t^2 + (y - x)^2} \, dy = 1 - \frac{2}{\pi} \arctan \frac{x}{t} \le \min \expr{1, \frac{t}{x}} \, .
}
For $t > 0$, $x, y \in D = (0, \infty)$, we have (see formula~(2.9) in~\cite{bib:bk04})
\formula{
  \frac{p_t(y - x) - p^D_t(x, y)}{t} & = \frac{1}{t} \, \ex_x \expr{p_{t - \tau_D}(y - X(\tau_D)) \; ; \; \tau_D \le t} \\
  & = \frac{1}{\pi t} \, \ex_x \expr{\frac{t - \tau_D}{(t - \tau_D)^2 + (y - X(\tau_D))^2} \; ; \; \tau_D \le t} \\
  & \le \frac{1}{\pi y^2} \, \pr_x \expr{\tau_D \le t} \le \min \expr{\frac{1}{\pi y^2}, \frac{t}{\pi x y^2}} .
}
By symmetry, also
\formula{
  \frac{p_t(y - x) - p^D_t(x, y)}{t} & \le \min \expr{\frac{1}{\pi x^2}, \frac{t}{\pi x^2 y}} .
}
Since $p_t(y - x) \le \frac{1}{\pi t}$, we conclude that
\formula[eq:C:pDt]{
  0 \le \frac{p_t(y - x) - p^D_t(x, y)}{t} & \le \frac{1}{\pi} \, \min \expr{\frac{1}{t^2}, \frac{1}{x^2}, \frac{1}{y^2}, \frac{t}{x^2 y}, \frac{t}{x y^2}} , && t,x,y > 0 .
}
This estimate is used in Sections~\ref{sec:pre} and~\ref{sec:hl}.

%
%

\section{Properties of $\eta$ and $B$}

\noindent
In Section~\ref{sec:hl}, a function $\eta$ being the generalized Hilbert transform of $-\arctan t_-$ is sought. More precisely, $\eta$ is the function satisfying $\eta(0) = 0$ and
\formula[eq:aux:int0]{
  \eta'(t) & = \frac{1}{\pi} \pv\int_{-\infty}^0 \frac{1}{(t - s)(1 + s^2)} \, ds , && t \in \R ,
}
the integral being the Cauchy principal value when $t < 0$. Observe that
\formula{
  \int \frac{1}{(t - s)(1 + s^2)} \, ds & = \frac{1}{1 + t^2} \int \expr{\frac{s + t}{1 + s^2} + \frac{1}{t - s}} ds \\
  & = \frac{1}{1 + t^2} \expr{t \arctan s + \frac{1}{2} \log(1 + s^2) - \log |t - s|} .
}
Hence we have
\formula{
  \eta'(t) & = \frac{1}{(1 + t^2)} \expr{\frac{t}{2} - \frac{1}{\pi} \log |t|}
}
and so
\formula[eq:aux:int1]{
  \eta(t) & = \frac{\log (1 + t^2)}{4} - \frac{1}{\pi} \int_0^t \frac{\log |s|}{1 + s^2} \, ds , && t \in \R .
}
In particular,
\formula[eq:aux:etasym]{
  \eta(-t) & = -\eta(t) + \log \sqrt{1 + t^2} , && t \in \R .
}
The integrals of $\frac{\log |s|}{1 + s^2}$ over $(0, \infty)$ and over $(-\infty, 0)$ are zero (this follows by a substitution $u = \frac{1}{s}$), and the maximum and minimum, equal to the Catalan constant $\mathcal{C} \approx 0.916$ and to $-\mathcal{C}$ respectively, is attained at $-1$ and $1$. It follows that
\formula[eq:aux:est]{
  \frac{1}{4} \log(1 + t^2) - \frac{\mathcal{C}}{\pi} \le \eta(t) & \le \frac{1}{4} \log(1 + t^2) + \frac{\mathcal{C}}{\pi} , && t \in \R ,
}
and in particular,
\formula[eq:aux:intest]{
  e^{\eta(t)} & \sim \sqrt{|t|} && \text{as } |t| \rightarrow \infty .
}
On the other hand, by~\eqref{eq:aux:int0},
\formula{
  \eta'(t) & = \frac{1}{\pi} \frac{d}{d t} \int_{-\infty}^0 \frac{\log |t - s|}{1 + s^2} \, ds,
}
and for $t = 0$,
\formula{
  \int_{-\infty}^0 \frac{\log |s|}{1 + s^2} \, ds & = \expr{\int_0^1 + \int_1^\infty} \frac{\log s}{1 + s^2} \, ds = \int_0^1 \frac{\log s}{1 + s^2} ds + \int_0^1 \frac{-\log s}{1 + s^{-2}} \, \frac{ds}{s^2} = 0 .
}
Therefore,
\formula[eq:aux:int2]{
  \eta(t) & = \frac{1}{\pi} \int_{-\infty}^0 \frac{\log |t - s|}{1 + s^2} \, ds , && t \in \R .
}
A related holomorphic function $B$ plays a major role in Sections~\ref{sec:hl}--\ref{sec:pdt}. It is defined by
\formula[eq:aux:B]{
  B(z) & = \frac{1}{\pi} \int_{-\infty}^0 \frac{\log (z - s)}{1 + s^2} \, ds , && z \in \C .
}
Here we agree that $\log(z) = \log |z| + i \frac{\pi}{2}$ for $z \in (-\infty, 0]$, i.e. $\log$ (and therefore also $B$) is continuous on $(-\infty, 0]$ when approached from $\C_+$, but not from $\C_-$; see also Section~\ref{sec:B}. The function $\real B(z)$ is harmonic in $\C \setminus (-\infty, 0]$, continuous in whole $\C$ and $\real B(t) = \eta(t)$ for $t \in \R$. For $z \in \C$, we have
\formula{
  \real B(z) & = \frac{1}{\pi} \int_{-\infty}^0 \frac{\log |z - s|}{1 + s^2} \, ds \le \frac{1}{\pi} \int_{-\infty}^0 \frac{\log (|z| - s)}{1 + s^2} \, ds = \eta(|z|) ,
}
and in a similar manner
\formula{
  \real B(z) & = \frac{1}{\pi} \int_{-\infty}^0 \frac{\log |z - s|}{1 + s^2} \, ds \ge \frac{1}{\pi} \int_{-\infty}^0 \frac{\log \bigl|-|z| - s\bigr|}{1 + s^2} \, ds = \eta(-|z|) .
}
By~\eqref{eq:aux:est},
\formula[eq:aux:Best]{
  \frac{1}{4} \log(1 + |z|^2) - \frac{\mathcal{C}}{\pi} \le \real B(z) & \le \frac{1}{4} \log(1 + |z|^2) + \frac{\mathcal{C}}{\pi} \, , && z \in \C .
}
In particular,
\formula[eq:aux:Bintest]{
  |e^{B(z)}| & \sim \sqrt{|z|} && \text{as } |z| \rightarrow \infty .
}
This estimates are used in Section~\ref{sec:hl} and in Section~\ref{sec:pdt} in contour integration. We also have $\imag B(t) = \arctan t_-$; this can be shown directly, or using the first part of this section as follows. The function $\real B'(t) = \eta'(t)$ is the Hilbert transform of $(-\arctan t_-)'$, and at the same time $\real B'(t)$ is the Hilbert transform of $-\imag B'(t)$, hence $\imag B'(t) = (\arctan t_-)'$. Since $\imag B(0) = 0 = \arctan 0_-$, we conclude that $\imag B(t) = \arctan t_-$.

The following auxiliary computations related to the functions $\eta$ and $B$ are used in Sections~\ref{sec:hl} and~\ref{sec:err}. We have
\formula{
  \int \frac{\frac{\pi}{2} - \arctan s}{1 + s^2} \, ds & = \frac{\pi}{2} \arctan s - \frac{1}{2} (\arctan s)^2 ,
}
so that
\formula[eq:aux:pi8]{
  \frac{1}{\pi} \int_0^\infty \frac{\frac{\pi}{2} - \arctan s}{1 + s^2} \, ds & = \frac{\pi}{8} \, .
}

By a substitution $s = \frac{1}{\tan t}$,
\formula{
  \int_0^\infty \frac{\log(1 + s^2)}{1 + s^2} \, ds & = -2 \int_0^{\frac{\pi}{2}} \log \sin t dt .
}
We have
\formula{
  2 \int_0^{\frac{\pi}{2}} \log \sin t dt & = \int_0^{\frac{\pi}{2}} \log \sin t dt + \int_0^{\frac{\pi}{2}} \log \cos t dt = \int_0^{\frac{\pi}{2}} \log \sin (2t) dt - \frac{\pi \log 2}{2} \\
  & = \frac{1}{2} \int_0^\pi \log \sin u du - \frac{\pi \log 2}{2} = \int_0^{\frac{\pi}{2}} \log \sin u du - \frac{\pi \log 2}{2} \, .
}
Therefore,
\formula[eq:aux:sqrt2]{
  \frac{1}{\pi} \int_0^\infty \frac{\log(1 + s^2)}{1 + s^2} \, ds & = \log 2 .
}
Whenever $a > -1$ and $b > \frac{1 + a}{2}$, we have by a substitution $1 + t^2 = \frac{1}{s}$ and a formula for the beta integral,
\formula[eq:aux:beta]{
  \int_0^\infty \frac{t^a}{(1 + t^2)^b} \, dt & = \frac{1}{2} \int_0^1 (1-s)^{\frac{a-1}{2}} s^{b - \frac{a+3}{2}} ds = \frac{\Gamma(\frac{a+1}{2}) \Gamma(b - \frac{a+1}{2})}{2 \Gamma(b)} \, .
}
Also, by integration by parts and $\Gamma(\frac{1}{2}) = \sqrt{\pi}$,
\formula[eq:aux:32]{
  \int_0^\infty \frac{1 - e^{-t x}}{t^{3/2}} \, dt & = 2 x \int_0^\infty \frac{e^{-t x}}{\sqrt{t}} \, dt = 2 \sqrt{\pi x} , && x > 0 .
}

%
%

\section{Estimates for the generator on a piecewise smooth function}

\noindent
The following estimates are used in Section~\ref{sec:int}. Define an auxiliary piecewise $C^2$ function:
\formula[eq:aux:q]{
  q(x) & = \begin{cases}
  0 & \text{for } x \in (-\infty, -\sfrac{1}{3}) , \\
  \sfrac{9}{2} (x + \sfrac{1}{3})^2 & \text{for } x \in (-\sfrac{1}{3}, 0) , \\
  1 - \sfrac{9}{2} (x - \sfrac{1}{3})^2 & \text{for } x \in (0, \sfrac{1}{3}) , \\
  1 & \text{for } x \in (\sfrac{1}{3}, \infty) .
  \end{cases}
}
Note that $q(x) + q(-x) = 1$. Let $f$ be a piecewise $C^2$ function on $\R$ and let $g(x) = q(x) f(x)$. Suppose that $g$ has a compact support. We estimate $\A g(x)$ for $x \in (-1, 0)$.

Choose $M_0$, $M_1$ and $M_2$ so that $|f(x)| \le M_0$, $|f'(x)| \le M_1$, $|f''(x)| \le M_2$ for $x \in (-\frac{1}{3}, \frac{1}{3})$. Let $I = \int_0^\infty |f(x)| dx$. Then
\formula{
  \abs{g''(x)} & \le M_0 \abs{q''(x)} + 2 M_1 \abs{q'(x)} + M_2 \abs{q(x)} \le 9 M_0 + 6 M_1 + M_2 .
}
If $z \in (-1, -\frac{1}{3})$, then $g(z) = 0$, and so $\A g(z)$ is estimated (up to the factor $\frac{1}{\pi}$) by
\formula{
  \int_{-\frac{1}{3}}^\infty \frac{\abs{g(x)}}{(x - z)^2} dx & \le M_0 \int_{-\frac{1}{3}}^{\frac{1}{3}} \frac{q(x)}{(x - z)^2} dx + \frac{9}{4} \int_{\frac{1}{3}}^{\infty} \abs{f(x)} dx \le 3 M_0 + \frac{9 I}{4} ;
}
here we used $\frac{q(x)}{(x - z)^2} \le \frac{9}{2}$ for $x \in (-\frac{1}{3}, \frac{1}{3})$ in the second inequality. For $z \in (-\frac{1}{3}, 0)$ the principal value integral in the definition of $\A$ can be estimated by splitting it into two parts. By Taylor's expansion of $g$, we have
\formula{
  \abs{\pv\int_{z-\frac{1}{3}}^{z+\frac{1}{3}} \frac{g(x) - g(z)}{(x - z)^2} dx} & \le \sfrac{2}{3} \cdot \sfrac{1}{2} \sup \set{\abs{g''(x)} \; : \; x \in (z-\sfrac{1}{3}, z+\sfrac{1}{3})} \\
  & \le \sfrac{1}{3} \sup \set{\abs{g''(x)} \; : \; x \in (-\sfrac{1}{3}, \sfrac{1}{3})} \le 3 M_0 + 2 M_1 + \sfrac{2 M_2}{3} ;
}
for the second inequality note that $g''(x) = 0$ for $x < -\frac{1}{3}$. Furthermore,
\formula{
  & \abs{\expr{\int_{-\infty}^{z-\frac{1}{3}} + \int_{z+\frac{1}{3}}^\infty} \frac{g(x) - g(z)}{(x - z)^2} dx} \\
  & \le |g(z)| \expr{\int_{-\infty}^{z-\frac{1}{3}} + \int_{z+\frac{1}{3}}^\infty} \frac{1}{(x - z)^2} dx + 9 \int_{z+\frac{1}{3}}^\infty |f(x)| dx \le 6 M_0 + 9 I .
}
We conclude that
\formula[]{
  \label{eq:aux:genest1} \abs{\A g(z)} & \le \frac{3 M_0 + \frac{9}{4} I}{\pi} \, , && z \in (-1, -\sfrac{1}{3}) ; \\
  \label{eq:aux:genest2} \abs{\A g(z)} & \le \frac{9 M_0 + 2 M_1 + \frac{2}{3} M_2 + 9 I}{\pi} \, , && z \in (-\sfrac{1}{3}, 0) .
}

%
%

\begin{acknowledgments}
The authors thank Nikolay Kuznetsov for pointing out the similarity of the spectral problem considered in this article and the sloshing problem.
\end{acknowledgments}

%
%

%
%

\end{document}